\documentclass[12pt]{amsart}
\usepackage{amssymb,amsfonts,amsmath,amsthm,cite,verbatim}
\usepackage{xcolor}
\usepackage{tikz}
\usepackage{graphicx}
\usepackage[left=2.5cm, right=2.5cm, top=3.5cm, bottom=3.5cm]{geometry}

\usepackage{hyperref}

\usepackage[normalem]{ulem}
\newtheorem{theorem}{Theorem}[section]
\newtheorem{thm}[theorem]{Theorem}
\newtheorem{prop}[theorem]{Proposition}
\newtheorem{lem}[theorem]{Lemma}
\newtheorem{rem}[theorem]{Remark}

\newtheorem{cor}[theorem]{Corollary}

\makeatletter \@addtoreset{equation}{section}

\newcommand{\red}[1]{{\color{red}  #1}}

\DeclareMathOperator*{\CT}{CT}

\def\CTA{\Xi}
\def\A{\mathcal{A}}
\def\x{\boldsymbol{x}}
\def\y{\boldsymbol{y}}

\def\id{\textrm{id}}
\def\inv{\mathrm{inv}}

\author{Guoce Xin$^{1}$, Chen Zhang$^{2}$, Yue Zhou$^{3,*}$ and Yueming Zhong$^{4}$}

\address{ $^{1, 2}$School of Mathematical Sciences,  Capital Normal University,
 Beijing 100048,  PR China,
 $^{3,4}$School of Mathematics and Statistics, HNP-LAMA, Central South University, Changsha
410083, P.R. China}
\email{$^1$\texttt{guoce\_xin@163.com}\\ $^2$\texttt{ch\_enz@163.com}\\ $^3$\texttt{zhouyue@csu.edu.cn}\\ \vskip 0.1cm $^4$\texttt{zhongyueming107@gmail.com}}

\thanks{$*$ Corresponding author.}
\begin{document}

\title[Constant term algebra]{The constant term algebra of type $A$: the structure}
%\subjclass[2010]{05A30, 33D70, 05E05}
\date{\today}
%\date{September 30, 2022}
%\begin{abstract}
%\noindent
%\textbf{Keywords:}
%\end{abstract}

\begin{abstract}
In this paper, we discover a new noncommutative algebra. We refer this algebra as the constant term algebra of type $A$, which is generated by certain constant term operators.
We characterize a structural result of this algebra by establishing an explicit basis in terms of certain forests. This algebra arises when we apply the method of the iterated Laurent series to investigate Beck and Pixton's residue computation for the Ehrhart series of the Birkhoff polytope. This algebra seems to be the first structural result in the area of the constant term world since the discovery of the Dyson constant term identity in 1962.
\end{abstract}
\maketitle

\noindent
\begin{small}
\end{small}

\noindent
\begin{small}
\end{small}

\section{Introduction}\label{s-introduction}

In the introduction,  we briefly review the development of the area of constant term identities.
Then we introduce our main contribution in this paper: the construction of the constant term algebra of type $A$.

The study of constant term identities originates from the Dyson constant term identity \cite{dyson1962} in 1962.
For nonnegative integers $a_1,\dots,a_n$,
\begin{equation}\label{thm-Dyson}
\CT_{\x} \,\prod_{1\leq i\neq j\leq n}(1-x_i/x_j)^{a_i}=\binom{a_1+\cdots +a_n}{a_1,\dots, a_n},
\end{equation}
where $\displaystyle\CT_{\x}f$ denotes taking the constant term of the Laurent polynomial (series) $f$ with respect to $\x:=(x_1,\dots,x_n)$,
and
\[
\binom{a_1+\cdots +a_n}{a_1,\dots, a_n}=\frac{(a_1+\cdots+a_n)!}{a_1!\cdots a_n!}
\]
is the multinomial coefficient.
Dyson's conjecture was soon proved by Gunson \cite{gunson} and Wilson \cite{wilson}.
Subsequently, using Lagrange interpolation Good \cite{good} discovered an elegant one-page proof, and much later, Zeilberger obtained a combinatorial proof using tournaments \cite{zeil}.

In 1975, Andrews \cite{andrews1975} conjectured a $q$-analogue of the Dyson constant term identity.
\begin{equation}\label{q-Dyson}
\CT_{\x} \prod_{1\leq i<j\leq n}
(x_i/x_j;q)_{a_i}(qx_j/x_i;q)_{a_j}=
\frac{(q;q)_{a_1+\cdots+a_n}}{(q;q)_{a_1}(q;q)_{a_2}\cdots(q;q)_{a_n}},
\end{equation}
where
$(z;q)_k:=(1-z)(1-zq)\cdots(1-zq^{k-1})$ is a
$q$-shifted factorial for a positive integer $k$ and $(z;q)_0:=1$.
Up till now, there exist four essentially different proofs of the $q$-Dyson identity \cite{zeil-bres1985,gess-xin2006,KN,cai}.

The equal parameter case of the $q$-Dyson identity leads to two directions.
One is the famous Macdonald constant term conjecture \cite{cherednik,macdonald82} and the theory of Macdonald symmetric function \cite{MacSMC,Mac95}. Another direction is the $q$-Morris identity and its equivalent $q$-Selberg integral \cite{Askey,Morris1982,Kadell1988,Habsieger1988}. Since the discovery of the Selberg integral \cite{Selberg}, many generalizations were found, such as \cite{AFLT11,ARW,BF,Kadell93,Kadell97,War09}. Forrester and Warnaar \cite{FW} gave a detailed review of the importance of the Selberg integral.

In the constant term world, many identities behave as bridges which connect different branches of mathematics, such as the Morris identity; or generate new theory, such as the Macdonald constant term conjecture; or just remain as items of the list. As far as we know, there exists no structural result yet. In this paper, we investigate a collection of certain constant term operators, and find that they form an algebra.

This work is inspired by the study of the Ehrhart polynomial $H_n(t)$ of the Birkhoff polytope \cite{BeckPixton,Louis}. The problem is equivalent to counting the number of $n\times n$ nonnegative integer matrices such that every row sum and every column sum equal to $t$.
The polynomial $H_n(t)$ has been studied by many researchers, e.g. \cite{BeckPixton,Clara,Diaconis,John,Per,Loera,Welleda}. The record is kept by Beck and Pixton \cite{BeckPixton}, who obtained an explicit expression for $H_9(t)$ using residue computation. When studying Beck and Pixton's residue method using iterated Laurent series, we discovered an algebra of constant term operators. We introduce this algebra below.

It is known that \cite{BeckPixton}
\[
H_n(t) = \CT_{\x,\y} \frac{\prod_{i=1}^n x_i^{-t} y_i^{-t} }{\prod_{1\le i,j\le n} (1-x_i y_j)}.
\]
By eliminating the $\y$ variables using partial fraction decompositions (PFD as abbreviation), we obtain
\begin{equation}\label{e-intro-1}
 H_n(t)=\CT_{\x} \prod_{i=1}^n x_i^{-t} \left(\sum_{i=1}^n \frac{x_i^t}{\prod_{j=1,j\neq i}^n (1-x_j/x_i)} \right)^n.
\end{equation}
Expanding the sum in \eqref{e-intro-1} by the multinomial theorem gives
\begin{equation}\label{e-intro-2}
H_n(t)= \sum \binom{n}{m_1,m_2,\dots, m_n} \CT_{\x} \prod_{i=1}^{n}\frac{ x_i^{(m_i-1)t}}{\prod_{j=1,j\neq i}^n (1-x_j/x_i)^{m_i}},
\end{equation}
where the sum is over all the nonnegative integers $m_i$ such that $\sum_{i=1}^nm_i=n$ for a fixed positive integer $n$.
Here and what follows, the rational functions in $\x$ are explained as their Laurent series expansions according to $1> x_1>\cdots>x_n>0$.

That writing $H_n(t)$ as in \eqref{e-intro-2} suggests us to consider the constant term of the next kind of rational functions of type $A$
\begin{equation}\label{e-intro-F}
F=\frac{L(\x)}{\prod_{1\leq i<j\leq n}(1-x_j/x_i)^{q_{ij}}},
\end{equation}
where $L(\x)$ is a homogeneous Laurent polynomial of degree 0 with coefficients in a field $\mathbb{F}$, the $q_{ij}$ are nonnegative integers,  and $F$ is written in a reduced form. Denote by $\mathcal{A}_n=\mathcal{A}_n(\x)$ the set of all such rational functions. We focus on the case when $L(\x)$ is a monomial. The general case
follows by linearity.

To describe our algebra, we introduce its basic building blocks: the constant term operator $\CT\limits_{x_i=x_j}$ and its slight extension $\CT\limits_{x_i=x_j^*}$.
For $u_1,\dots,u_m\in \{1,2,\dots,n\}$ such that $u_r\neq i$,
$c_r\in \mathbb{C}$ and $q_r$ positive integers, let
\begin{equation}\label{e-intro-3}
  f(x_i)=\frac{L_0(x_i)}{\prod_{r=1}^m(1-c_rx_i/x_{u_r})^{q_r}},
\end{equation}
where $L_0(x_i)$ is a Laurent polynomial in $x_i$, and $c_r\neq c_s$ if $x_{u_r}=x_{u_s}$.
Assume the PFD of $f$ with respect to $x_i$ is given by
\begin{equation}\label{e-intro-4}
f(x_i)=L_1(x_i)+\sum_{r=1}^m\frac{A_r(x_i)}{(1-c_rx_i/x_{u_r})^{q_r}},
\end{equation}
where $L_1(x_i)$ is a Laurent polynomial, and $A_r(x_i)$ is a polynomial of degree less than $q_r$ for each $r$.
Define
\begin{equation}\label{e-intro-CTij}
\CT_{x_i=c_r^{-1}x_{u_r}} f(x_i)=A_r(0) \quad \text{and}  \quad \CT_{x_i=x_j^*}f(x_i)= \sum_{u_r=j}A_r(0).
\end{equation}
In fact, we can write
\[
\CT_{x_i=x_j^*}f(x_i)= \sum_{u_r=j}\CT_{x_i=c_r^{-1}x_{u_r}} f(x_i)
\]
by the above definition.
In particular, if there is only one $r$ such that $u_r=j$, then
$\CT\limits_{x_i=x_j^*}f=\CT\limits_{x_i=c_r^{-1}x_j}f$.
Furthermore, if $c_r=1$ then $\CT\limits_{x_i=x_j^*}f=\CT\limits_{x_i=x_j}f$.
We define the general operator $\CT\limits_{x_i=x_j^*}$ in \eqref{e-intro-CTij} for simplicity of some proofs in this paper and further development of this theory in the future.
We remark that when acting on $\A_n$, we only need the operator $\CT\limits_{x_i=x_j}$.

In this paper, we observe that such constant terms $\CT\limits_{x_i=x_j}$ as operators on $\A_n$ generate an algebra. We call it the constant term algebra of type $A$ and denote by $\CTA$. This algebra characterizes the structure of certain constant term operators.
In this paper, we also discover a basis of $\CTA$, as stated in Theorem~\ref{thm-struct} below.
In the light of $\CTA$, we develop a new method to compute constant terms of type $A$ rational functions.
In the next sequel, we will use $\CTA$ to simplify the computation of $H_n(t)$. It is hopeful to obtain an explicit formula for $H_{10}(t)$.

The structure of this paper is as follows.
%Section 1 is this introduction.
In Section 2, we introduce notation and basic facts about $\CTA$.
In Section 3, we establish the structure of $\CTA$ by giving an explicit basis.

\section{Notation and basic facts}\label{s-notation}

In this section, we give notations used throughout this paper and introduce several basic facts
of $\CTA$.

We obtain an explicit expression for the PFD of the rational function $f(x_i)$
in \eqref{e-intro-3} using differential operators.
\begin{lem}\label{l-A10}
Let $f(x_i)$ and $A_r(x_i)$ be as in \eqref{e-intro-3} and \eqref{e-intro-4} respectively. For $r\in \{1,2,\dots,m\}$, denote
\begin{equation}\label{e-defin-g}
g_r(x_i)=f(x_i)(1-c_rx_i/x_{u_r})^{q_r}=\frac{L_0(x_i)}{\prod_{l=1,l\neq r}^m(1-c_lx_i/x_{u_l})^{q_l}}.
\end{equation}
Then
\begin{equation}\label{e-u-A1}
A_r(x_i)=\sum_{k=0}^{q_r-1} \frac{(-1)^k\big(D_w^k g_r(wc_r^{-1}x_{u_r})\big)\big|_{w=1}}{k!}\cdot(1-c_rx_i/x_{u_r})^{k},
\end{equation}
where $D_w$ is the differential operator with respect to $w$.
In particular,
\begin{equation}\label{e-u-A10}
\CT_{x_i=c_r^{-1}x_{u_r}}f(x_i)= A_r(0)= \frac{(-1)^{q_r-1}}{(q_r-1)!}\Big(D_w^{q_r-1}\frac{g_r(wc_r^{-1}x_{u_r})}{w}\Big)\Big|_{w=1}.
\end{equation}
Furthermore, if $q_r=1$ then
\begin{equation}\label{e-struc-1}
A_r(0)=g_r(c_r^{-1}x_{u_r}).
\end{equation}
\end{lem}
\begin{proof}
We only prove the $r=1$ case. The other cases can be obtained by symmetry.

Observe that
\[
\frac{(-1)^{q_1-1}}{(q_1-1)!} D_w^{q_1-1}\Big(\frac{1}{w-c_1x_i/x_{u_1}}\Big)= \frac{1}{(w-c_1x_i/x_{u_1})^{q_1}}.
\]
Using this we can write
\begin{equation}\label{e-struc-2}
f(x_i)=\frac{(-1)^{q_1-1}}{(q_1-1)!}\Big(D_w^{q_1-1}\frac{L_0(x_i)}
{(w-c_1x_i/x_{u_1})\prod_{r=2}^m(1-c_rx_i/x_{u_r})^{q_r}}\Big)\Big|_{w=1}.
\end{equation}
The PFD of the next rational function with respect to $x_i$ is given by
\begin{equation}\label{e-struc-3}
\frac{L_0(x_i)}{(w-c_1x_i/x_{u_1})\prod_{r=2}^m (1-c_rx_i/x_{u_r})^{q_r}}
=\overline{L}_1(x_i,w)+ \frac{\overline A_1(x_i,w)}{(w-c_1x_i/x_{u_1})}
+\sum_{r=2}^m \frac{\overline A_{r}(x_i,w)}{(1-c_rx_i/x_{u_r})^{q_r}},
\end{equation}
where
\[
\overline A_1(x_i,w) = \frac{L_0(x_i)}{\prod_{r=2}^m (1-c_rx_i/x_{u_r})^{q_r}}\Big|_{x_i=wc_1^{-1}x_{u_1}} = g_1(wc_1^{-1}x_{u_1}),
\]
and $\overline{L}_1(x_i,w)$ and $\overline A_r(x_i,w)$ are Laurent polynomial and polynomial in $x_i$ but rational functions in $w$ respectively.
Substituting \eqref{e-struc-3} into \eqref{e-struc-2} yields
\begin{multline*}
f(x_i)=\frac{(-1)^{q_1-1}}{(q_1-1)!} \Big(D_w^{q_1-1}\frac{g_1(wc_1^{-1}x_{u_1})}{w-c_1x_i/x_{u_1}}\Big)\Big|_{w=1} \\ +\frac{(-1)^{q_1-1}}{(q_1-1)!}\Big(D_w^{q_1-1}\overline{L}_1(x_i,w)+D_w^{q_1-1}\sum_{r=2}^m \frac{\overline A_{r}(x_i,w)}{(1-c_rx_i/x_{u_r})^{q_r}} \Big)\Big|_{w=1}.
\end{multline*}
By the uniqueness of the PFD of $f(x_i)$,
\begin{equation}\label{e-struc-4}
\frac{A_1(x_i)}{(1-c_1x_i/x_{u_1})^{q_1}}=\frac{(-1)^{q_1-1}}{(q_1-1)!} \Big(D_w^{q_1-1}\frac{g_1(wc_1^{-1}x_{u_1})}{w-c_1x_i/x_{u_1}}\Big)\Big|_{w=1}.
\end{equation}
We can further write
\begin{align}
D_w^{q_1-1}\frac{g_1(wc_1^{-1}x_{u_1})}{w-c_1x_i/x_{u_1}}
&=\sum_{k=0}^{q_1-1}\binom{q_1-1}{k} D_w^{k}g_1(wc_1^{-1}x_{u_1})\cdot
D_w^{q_1-k-1}(w-c_1x_i/x_{u_1})^{-1}\nonumber\\
&=\sum_{k=0}^{q_1-1}(-1)^{q_1-k-1}(q_1-k-1)! \binom{q_1-1}{k}
\frac{D_w^{k}g_1(wc_1^{-1}x_{u_1})}
{(w-c_1x_i/x_{u_1})^{q_1-k}}.\label{e-struc-5}
\end{align}
Substituting \eqref{e-struc-5} into \eqref{e-struc-4} yields \eqref{e-u-A1} for the $r=1$ case.

Taking $x_i=0$ in \eqref{e-struc-4} gives \eqref{e-u-A10} for the $r=1$ case. That is the formula for $A_1(0)$.

Taking $q_r=1$ in \eqref{e-u-A10} yields \eqref{e-struc-1}.
\end{proof}

\begin{rem}\label{rem-Laurent-view}
The view of the operator $\CT\limits_{x_i=c_r^{-1}x_{u_r}}$
using Laurent series may be helpful.

Expand $f(x_i)$ as a Laurent series at $x_i=c_r^{-1}x_{u_r}$,
\[
f(x_i)=\sum_{k=-q_r}^{\infty} h_k(c_r) (1-c_r x_i/x_{u_r})^k,
\]
where the $h_k(c_r)$ are rational functions free of $x_i$. Comparing this with \eqref{e-intro-4}, we see that
\[
\frac{A_r(x_i)}{(1-c_rx_i/x_{u_r})^{q_r}}=\sum_{k=-q_r}^{-1} h_k(c_r) (1-c_r x_i/x_{u_r})^{k}.
\]
It follows that
$$\CT_{x_i=c_r^{-1}x_{u_r}} f(x_i) =A_r(0)= \sum_{k=-q_r}^{-1} h_k(c_r).$$
In particular,
\[
\CT_{x_i=c_r^{-1}x_{u_r}} (1-c_r x_i/x_{u_r})^k
=\begin{cases}
1, \quad \text{if $k<0$;}\\
0, \quad \text{if $k\ge 0$.}
\end{cases}
\]
\end{rem}

If all the $c_r=1$, the rational function $f(x_i)$ in \eqref{e-intro-3} is a rational function of type $A$ in \eqref{e-intro-F}, i.e.,
\[
f(x_i)=\frac{L_0(x_i)}{\prod_{r=1}^{m} (1-x_i/x_{u_r})^{q_r} },
\]
where the distinct integers $u_r\in \{1,\dots,n\}$ are not equal to $i$.
By \eqref{e-intro-4},
\begin{equation}\label{e-sec2-1}
f(x_i)=L_1(x_i)+\sum_{r=1}^{m} \frac{A_r(x_i)}{(1-x_i/x_{u_r})^{q_r}}.
\end{equation}
We have an expression for $A_r(x_i)$ in \eqref{e-u-A1}.
But the explicit formula for $A_r(x_i)$ without the slack variable $w$ is quite complicated.
Nevertheless, we can obtain its denominator. That is the next result.
\begin{cor}\label{cor-f-typeA}
The denominator of $A_r(x_i)$ in \eqref{e-sec2-1} is $\prod_{l=1, \;l\neq r}^{m} (1-x_{u_r}/x_{u_l})^{q_l+q_r-1}$.
\end{cor}
Note that $A_r(x_i)$ is only a polynomial in $x_i$, but it may be a rational function
in other variables. By Corollary~\ref{cor-f-typeA} we can know that the denominator of $A_r(0)$ is also $\prod_{l=1, \;l\neq r}^{m} (1-x_{u_r}/x_{u_l})^{q_l+q_r-1}$.
\begin{proof}
We only prove the $r=1$ case. The other cases can be obtained by symmetry.

By Lemma \ref{l-A10},
$$A_1(x_i)=\sum_{k=0}^{q_1-1} \frac{(-1)^k}{k!}
\Big(D_w^k \frac{L_0(wx_{u_1})}{\prod_{l=2}^{m} (1-wx_{u_1}/x_{u_l})^{q_l} }\Big)\Big|_{w=1}\cdot(1-x_i/x_{u_1})^{k}.$$
It is clear that $\prod_{l=2}^{m} (1-x_{u_1}/x_{u_l})^{q_l+k}$
is the denominator of
$\big(D_w^k \frac{L_0(wx_{u_1})}{\prod_{l=2}^{m} (1-wx_{u_1}/x_{u_l})^{q_l} }\big)\big|_{w=1}$.
Then the corollary follows.
\end{proof}

We give the following simple result, which will be used in the next section.
\begin{cor}\label{cor-abc}
The PFD of the next rational function with respect to $x$ is giving by
$$ \frac{1}{(1-x/y)^{a}(1-x/z)^b (1-y/z)^c} = \frac{L_1(x,y,z)}{(1-x/y)^a(1-y/z)^{a+b+c-1}} +\frac{L_2(x,y,z)}{(1-x/z)^b (1-y/z)^{a+b+c-1}},$$
where $a, b, c$ are nonnegative integers, and $L_1(x,y,z), L_2(x,y,z)$ are Laurent polynomials in $x,y,z$.
\end{cor}
\begin{proof}
We can write
$$ \frac{1}{(1-x/y)^{a}(1-x/z)^b (1-y/z)^c} = \frac{A_1(x)}{(1-x/y)^a} +\frac{A_2(x)}{(1-x/z)^b}$$
by the PFD with respect to $x$.
Here we have explicit formulas for $A_1$ and $A_2$ by Lemma~\ref{l-A10}.
Using \eqref{e-u-A1} in Lemma~\ref{l-A10}, we have
\[
A_1(x)= \sum_{k=0}^{a-1} \frac{(-1)^k}{k!}\Big(D_w^k \frac{1}{(1-wy/z)^b (1-y/z)^c}\Big)\Big|_{w=1}\cdot(1-x/y)^{k}.
\]
It is clear that $(1-y/z)^{b+c+k}$ is the denominator of $\big(D_w^k \frac{1}{(1-wy/z)^b (1-y/z)^c}\big)\big|_{w=1}$.
Then $(1-y/z)^{a+b+c-1}$ is the denominator of $A_1$. It follows that $\frac{A_1(x)}{(1-x/y)^a}$ is of the form
\[
\frac{L_1(x,y,z)}{(1-x/y)^a(1-y/z)^{a+b+c-1}}
\]
in the corollary.
The formula for $A_2$ is similar.
\end{proof}

Using \eqref{e-intro-CTij} and \eqref{e-u-A10} we can obtain the next result for computing
the constant term $\CT\limits_{x_i=x_j^*}f(x_i)$.
\begin{lem}
Let $f(x_i)$, $g_r(x_i)$ and $\CT\limits_{x_i=x_j^*}$ be defined in \eqref{e-intro-3},\eqref{e-defin-g} and \eqref{e-intro-CTij} respectively.
Then
\begin{equation}\label{e-struc-6-0}
\CT_{x_i=x_j^*}f(x_i)=\sum \frac{(-1)^{q_r-1}}{(q_r-1)!}\Big(D_w^{q_r-1}\frac{g_r(wc_r^{-1}x_{u_r})}{w}\Big)\Big|_{w=1},
\end{equation}
where the sum here and in the next equation is over all $r$ such that $u_r=j$.
In particular, if all the $q_r$ in the above sum equal $1$, then
\begin{equation}\label{e-struc-6}
\CT_{x_i=x_j^*}f(x_i)=\sum g_r(c_r^{-1}x_{u_r}).
\end{equation}
\end{lem}

Recall the definition of the constant term operator $\CT\limits_{x_i=x_j^*}$ in \eqref{e-intro-CTij}. It reduces to $\CT\limits_{x_i=x_j}$ when acting on $\A_n$.
We denote $\CT\limits_{x_i=x_j}$ by $[i,j]$ for simplicity and refer it as a commutator. The constant term algebra $\CTA$ of type $A$ is generated by all such commutators. The algebra $\CTA$ is a subalgebra of the algebra consisting of all linear transformations on $\A_n$.

\def\L{\mathbf{L}}

An operator $\L\in \CTA$ is $0$ if and only if $\L\circ F=\L(F)=0$ for all $F\in \A_n$.
The scalar multiplication, sum, and multiplication of operators are naturally defined
for $\L_1,\L_2\in \CTA$ and $F\in \A_n$ as
\begin{subequations}
\begin{equation}
(k\cdot \L_1)\circ F=k\big(\L_1(F)\big) \quad \text{for} \quad k\in \mathbb{C},
\end{equation}
\begin{equation}
(\L_2+\L_1)\circ F=\L_2(F)+\L_1(F),
\end{equation}
and
\begin{equation}
(\L_2 \cdot \L_1)\circ F=\L_2 \big(\L_1(F)\big)
\end{equation}
respectively.
\end{subequations}
We simply write $\L(F)=\L F$ if there is no ambiguity.
We say a monomial operator $L=[i_s,j_s]\cdots [i_1,j_1]$ is of degree $s=\deg(L)$ if it is nonzero. The identity operator $\id$ is of degree $0$. Clearly the product $L_1\cdot L_2$ of two monomial operators is either $0$ or of degree $\deg(L_1)+\deg(L_2)$.

The algebra $\CTA$ of type $A$ is a graded algebra with a natural decomposition
\[
\CTA=\biguplus_{s=0}^{n-1} \CTA^s,
\]
where $\CTA^s$ is the vector space spanned by all monomial operators of degree $s$. We will show that
this is in fact a direct sum decomposition and that $\CTA^s$ is the null space when $s\ge n$.

The $n=1$ case is worth mentioning. Type $A$ rational functions in one variable are indeed complex numbers, i.e., $\A_1=\mathbb{C}$:
by definition each $F\in \A_1$ is a homogeneous Laurent polynomial in $x_1$ of degree $0$, and hence a complex number. Consequently, $\CTA^0$ is spanned by the identity operator $\id$.

In what follows, a graph $G=(V,E)$ is always a simple (vertex) labelled graph unless specified otherwise, where $V$ and $E$ are the vertex set and the edge set of $G$ respectively.
If $(u,v)\in E$, let $G/{u\to v}$ be the graph obtained from $G$ by identifying $u$ with $v$ (or contracting $u$ to $v$).
Define a subset of $\A_n$ according to $H$ (a subgraph of the complete graph $K_n$) as
\[
\A_n(H)=\Big\{F: F=\frac{L(\x)}{\prod_{1\leq i<j\leq n}(1-x_j/x_i)^{q_{ij}}}\in \A_n,
q_{ij}=0 \text{ if } (i,j) \notin E(H)\Big\}.
\]
We call $x_i=x_j$ ($x_j=x_i$) a nonzero pole of $F$ with multiplicity $q_{ij}$ if $q_{ij}>0$.
More generally, we say that $x_i=dx_j$ ($x_j=d^{-1}x_i$) is a pole of a reduced rational function $R$
if $1-x_i/(dx_j)$ is a factor of the denominator of $R$ for $d\in \mathbb{C}\setminus \{0\}$.
Then it is obvious to obtain the next simple result (by Corollary~\ref{cor-f-typeA}).
\begin{lem}\label{lem-struc-1}
Let $F\in \A_n$, $[i,j]\in \CTA$, and $H$ be a subgraph of $K_n$.
\begin{enumerate}
\item If $x_i=x_j$ is not a pole of $F$, then $[i,j]F=0$.

\item The action of $[i,j]$ on $\A_n$ is closed. That is, $[i,j]F\in \A_n$. Explicitly,
$[i,j]$ eliminates $x_i$ and takes $\A_n(H)$ to $\A_n(H/{i\to j})$.
Thus, if a nonzero monomial operator $L\in \CTA$ contains $[i,j]$, then $i$ can not appear to the left of $[i,j]$ in $L$.
\end{enumerate}
 \end{lem}

We will identify a monomial operator with its ordered digraph.

 For a monomial operator $L=[i_s,j_s]\cdots [i_1,j_1]\in \CTA$, define an ordered digraph $D(L)$ of $L$ as:  the vertex set of $D(L)$ is $\{i_1,j_1,\dots,i_s,j_s\}$, which is a subset of $\{1,2,\dots,n\}$, and the edge set contains all the directed edges $\{i_k\rightarrow j_k: 1\leq k\leq s\}$. Furthermore, for each inner vertex $j$, if several edges direct to $j$, then we put these edges from right to left as their corresponding order in $L$. That is,  ``ordered'' means that the children of every inner vertex are arranged from right to left according to their order in $L$. See {Figure~\ref{fig-DL-ex}} for an example.
If all the edges of a connected digraph direct towards a vertex $r$, then we say $r$ is the root. Having the above definition, we give the next corollary of Lemma~\ref{lem-struc-1}.

\begin{figure}[htpb]
  \centering
  \includegraphics[width=6 cm]{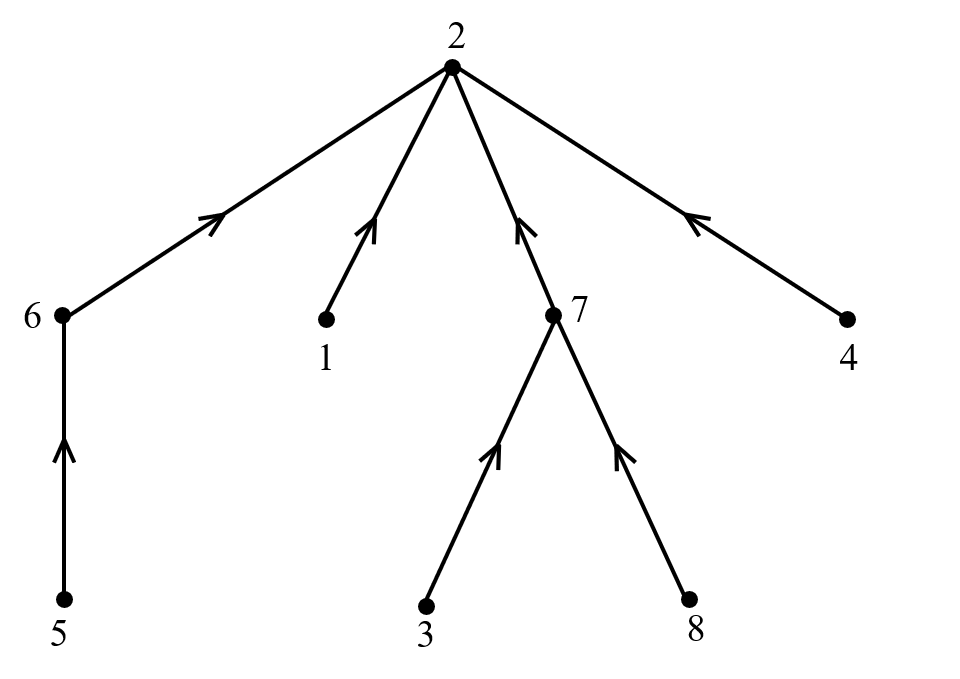}\\
  \caption{$D(L)$ of $L=[6,2][5,6][1,2][7,2][3,7][8,7][4,2]$.}\label{fig-DL-ex}
\end{figure}

\begin{cor}\label{cor-basic}
 Let $L=[i_s,j_s]\cdots [i_1,j_1]\in \CTA$. Then
\begin{enumerate}
 \item The monomial operator $L$ is nonzero only when $s\le n-1$ and all the $i$'s are distinct. Then, $\CTA^{s}$ is the null space for $s\geq n$, and each vertex of $D(L)$ has out degree at most $1$.

 \item The operator $L$ is nonzero only when its digraph $D(L)$ is a directed rooted forest.
\end{enumerate}
 \end{cor}
\begin{proof}
(1) Assume two $i$'s in $L$ are identical, say $i_a=i_b$ for  $1\le a <b\le s$. For any $F\in \A_n$, by the part (2) of Lemma~\ref{lem-struc-1} we have $L(F)=0$.

If $s>n$, then there exist two identical $i$'s in $L$ since all the $i\in \{1,2,\dots,n\}$. By the discussion above, we have $L(F)=0$.

If $s=n$ and all the $i$'s are distinct, then $(i_1,\dots,i_n)$ is a permutation of $\{1,2,\dots,n\}$. For any $F\in \A_n$,
we have $E=[i_{n-1},j_{n-1}]\cdots [i_1,j_1]F\in \mathbb{C}$ since $F$ is homogeneous in $x_1,\dots,x_n$ of degree 0.
Applying $[i_n,j_n]$ to $E$ yields 0.

If a vertex of $D(L)$ has out degree at least two, then there exist two identical $i$'s in $L$. By the discussion above, $L=0$.

(2) By the result that out degree of every vertex is at most 1 in the part (1), $D(L)$ must be rooted. Then, it suffices to show that $D(L)$ contains no cycle.
Assume to the contrary that $D(L)$ contains a cycle, then the cycle must be a directed cycle since
out degree of every vertex is at most 1 (again by the part (1)).
Denote the directed cycle of $D(L)$ by
$[i_{s'},j_{s'}], \dots, [i_{2'},j_{2'}], [i_{1'},j_{1'}]$ with $1\le 1'<2'<\cdots<s'\le s$
and $\{i_{1'},\dots, i_{s'}\}=\{j_{1'},\dots, j_{s'}\}$.
That is, there exits an $l'\in \{2',\dots,s'\}$ such that $j_{l'}=i_{1'}$.
By the part (2) of Lemma~\ref{lem-struc-1}, we have $L=0$, a contradiction.
\end{proof}

\section{The structure of $\CTA$}\label{s-structure}

In this section, we obtain a structural result of the constant term algebra of type $A$.
That is, we characterize the basis of $\CTA$.

\subsection{The rules in $\CTA$}

In this subsection, we mainly introduce three important rules in $\CTA$ for constructing its basis: the Commutativity Rule, the Exchange Rule and the V-Formula.
Equation \eqref{e-struc-6-0} is our basic tool for proving these rules.

Let us begin with the next simple result, which will be used in constructing the basis of $\CTA^{n-1}$.
\begin{lem}\label{lem-extrachange}
Let $F$ be a rational function in two variables of the form
\[
F=\frac{L(x_i,x_j)}{(1-x_i/x_j)^m},
\]
where $L(x_i,x_j)$ is a homogeneous Laurent polynomial in $x_i$ and $x_j$ of degree 0 with coefficients in $\mathbb{C}$, and $m$ is a nonnegative integer.
Then
\begin{equation}\label{e-struc-9}
[i,j]F=-[j,i]F\in \mathbb{C}.
\end{equation}
\end{lem}
\begin{proof}
For $m=0$, both sides of \eqref{e-struc-9} vanish by the part (1) of Lemma~\ref{lem-struc-1}.
Hence, we assume $m$ to be a positive integer in the following proof.
Since $L(x_i,x_j)$ is homogeneous in $x_i$ and $x_j$ of degree 0, we may assume $L=(x_i/x_j)^k$ for an integer $k$  by linearity.
By \eqref{e-struc-6-0},
\[
[i,j]F=[i,j]\frac{(x_i/x_j)^k}{(1-x_i/x_j)^m}=\bigg(\frac{(-D_w)^{m-1}}{(m-1)!}w^{k-1}\bigg)\bigg|_{w=1}
=\frac{(-1)^{m-1}}{(m-1)!}(k-1)(k-2)\cdots (k-m+1),
\]
and
\begin{multline*}
[j,i]F= [j,i]\frac{(-1)^m (x_i/x_j)^{k-m}}{(1-x_j/x_i)^m}
= \bigg(\frac{(-D_w)^{m-1}}{(m-1)!}(-1)^m w^{-k+m-1}\bigg)\bigg |_{w=1}\\
= \frac{(-1)^{m-1}}{(m-1)!}(-1)^m (-k+m-1)\cdots (-k+2)(-k+1)
=\frac{(-1)^{m}}{(m-1)!}(k-1)(k-2)\cdots (k-m+1).
\end{multline*}
Then \eqref{e-struc-9} holds and the lemma follows.
\end{proof}

To find the relationship between the operators of $\CTA$, we need a formula for $[i',j'][i,j] F$.
Equation \eqref{e-struc-6-0} provides an explicit formula for $[i,j]F$ involving the slack variable $w$ (set $w=1$ at last) and the differential operator $D_w$.
It is not wise to remove $w$ before applying $[i',j']$ to $[i,j]F$. Because the expression for $[i,j]F$ after removing
$w$ is quite complicated.
Thus, it is necessary to work on parameterized type $A$ rational functions.
Denote by $ \A_n[u,u^{-1}]$ the ring of Laurent polynomial of $u$ with coefficients in $\A_n$.
It is natural to obtain the following switch rules.
\begin{lem}\label{lem-exchange}
 Let $u$ be a parameter in $\mathbb{C}$, and let $f,g$ be rational functions of the form
\[
f=\frac{E(u,x_i,x_j)}{(u-x_i/x_j)^q}, \qquad  g=\frac{E(u,x_i,x_j)}{(1-ux_i/x_j)^q},
\]
where $q\in \mathbb{Z}$ and $x_i= x_j$ is not a pole of $E\in \A_n[u,u^{-1}]$. Then
\begin{align}
  D_u\CT_{x_i=x_j^*}f &=\CT_{x_i=x_j^*}D_u f,   \label{e-(1)} \\
  \CT_{x_i=x_j}(g |_{u=1})&=\Big(\CT_{x_i=x_j^*} g\Big)\Big|_{u=1}. \label{e-(2)}
\end{align}
\end{lem}
\begin{proof}
The case $q\le 0$ is trivial since both sides of \eqref{e-(1)} and \eqref{e-(2)} vanish.
So we assume $q>0$.

By Remark \ref{rem-Laurent-view},
we can write
$$ f = \sum_{k=-q}^{\infty} h_k(u) \big(1-x_i/(ux_j)\big)^{k}.$$
% and linearity, we may assume $E(u,x_i,x_j) = u^s (1-x_i/x_j)^t$ for some integers $s$ and $t$.
%And $t$ is indeed nonnegative since $x_i=x_j$ is not a pole of $E$.

%Expand $E(u,x_i,x_j)$ as a Laurent series at $x_i=x_j$. This is indeed a power series since $x_i=x_j$ is not a pole of $E$.
%By linearity, we may assume $E(u,x_i,x_j)=u^sx_i^t$ for some integers $s$ and $t$. %\red{???}

To prove \eqref{e-(1)}, we compute directly:
\[
D_u \CT_{x_i=x_j^*} f =D_u \CT_{x_i=ux_j} f= D_u \sum_{k=-q}^{-1} h_k(u) = \sum_{k=-q}^{-1} h_k'(u),
\]
and
\begin{align*}
\CT_{x_i=x_j^*}D_u f
&= \CT_{x_i=ux_j} \sum_{k=-q}^{\infty} \Big(h_k'(u)\big(1-x_i/(ux_j)\big)^k
+ k u^{-2} h_k(u)  (x_i/x_j)\big(1-x_i/(ux_j)\big)^{k-1}\Big) \\
&= \sum_{k=-q}^{-1} h_k'(u)+\sum_{k=-q}^{\infty}k u^{-2} h_k(u)\CT_{x_i=ux_j}(x_i/x_j)\big(1-x_i/(ux_j)\big)^{k-1}=\sum_{k=-q}^{-1} h_k'(u).
\end{align*}
Here the last equation holds since
$\CT\limits_{x_i=ux_j}(x_i/x_j)\big(1-x_i/(ux_j)\big)^{k-1}=0$
($x_i=ux_j$ is not a pole for $k>0$ and using the definition \eqref{e-intro-CTij} directly for $k<0$).

To prove \eqref{e-(2)}, using Remark~\ref{rem-Laurent-view} again, we can write
\[
  g  =\frac{E(u,x_i,x_j)}{(1-ux_i/x_j)^q}= \sum_{k=-q}^{\infty} \tilde{h}_k(u) (1-ux_i/x_j)^{k},
\]
and
\[
g|_{u=1}=\frac{E(1,x_i,x_j)}{(1-x_i/x_j)^q}=\sum_{k=-q}^{\infty} \overline{h}_k (1-x_i/x_j)^{k}.
\]
Then $\tilde{h}_k(1)=\overline{h}_k$ exists for each $k$, and we obtain
\[
\CT_{x_i=x_j}(g |_{u=1})= \CT_{x_i=x_j} \sum_{k=-q}^{\infty} \tilde{h}_k(1) (1-x_i/x_j)^{k} = \sum_{k=-q}^{-1} \tilde{h}_k(1),
\]
and
\[
\Big(\CT_{x_i=x_j^*} g\Big)\Big|_{u=1} =\Big(\CT_{x_i=u^{-1}x_j} g\Big)\Big|_{u=1} = \Big(\sum_{k=-q}^{-1} \tilde{h}_k(u)\Big)\Big|_{u=1} = \sum_{k=-q}^{-1} \tilde{h}_k(1).
\]
Note that this argument does not work if $E$ has a pole at $x_i=x_j$.
\end{proof}

\begin{rem}
Equation \eqref{e-(2)}
also holds if $g$ is multiplied by $(1-x_i/x_j)^{-q'}$ for another integer $q'$, but it seems hard to verify that
directly.
\end{rem}

Now we introduce the Commutativity Rule.
\begin{lem}[Commutativity Rule]\label{lem-comm}
Suppose $1\le i,j,k,l \le n$ are distinct integers. Then
\[
[k,l][i,j]=[i,j][k,l], \quad i.e., \quad \CT_{x_k=x_l}\CT_{x_i=x_j}F=\CT_{x_i=x_j}\CT_{x_k=x_l}F \text{ for all } F\in \A_n.
\]
 \end{lem}
\begin{proof}
For any $F\in \A_n$, suppose $x_i=x_j$ and $x_k=x_l$ are poles of $F$ with multiplicities $m_1$ and $m_2$ respectively. Then, we can write
\begin{align*}
F&= \frac{E(x_i,x_j,x_k,x_l)}{(1-x_i/x_j)^{m_1}(1-x_k/x_l)^{m_2}},
%&=\bigg(\frac{(-D_u)^{m_1-1}}{(m_1-1)!} \frac{(-D_v)^{m_2-1}}{(m_2-1)!}
%\frac{1}{(u-x_i/x_j)(v-x_k/x_l)}E(x_i,x_j,x_k,x_l)\bigg)\bigg|_{u=v=1},
\end{align*}
where $E\in \A_n$ may contain other variables but $x_i=x_j$ and $x_k=x_l$ are not poles of $E$.
By applying the constant term operator $[i,j]$ to $F$ and using \eqref{e-struc-6-0},
\begin{align*}
[i,j]F&= \frac{(-D_u)^{m_1-1}}{(m_1-1)!} \frac{E(ux_j,x_j,x_k,x_l)}{u(1-x_k/x_l)^{m_2}} \Big |_{u=1}\\
&= \frac{1}{(1-x_k/x_l)^{m_2}} \left(\frac{(-D_u)^{m_1-1}}{(m_1-1)!} \frac{E(ux_j,x_j,x_k,x_l)}{u} \Big |_{u=1}\right).
\end{align*}
Next applying $[k,l]$ to both sides of the above and using \eqref{e-struc-6-0} again, we have
\begin{align*}
[k,l][i,j]F&= [k,l]\frac{1}{(1-x_k/x_l)^{m_2}} \left(\frac{(-D_u)^{m_1-1}}{(m_1-1)!} \frac{E(ux_j,x_j,x_k,x_l)}{u} \Big |_{u=1}\right) \\
&= \frac{(-D_v)^{m_2-1}}{(m_2-1)!}\frac{1}{v}\left(\frac{(-D_u)^{m_1-1}}{(m_1-1)!} \frac{E(ux_j,x_j,vx_l,x_l)}{u} \Big |_{u=1}\right)\Big |_{v=1} \\
&= \frac{(-D_u)^{m_1-1}}{(m_1-1)!} \frac{(-D_v)^{m_2-1}}{(m_2-1)!} \frac{E(ux_j,x_j,vx_l,x_l)}{uv}\Big|_{u=v=1}.
\end{align*}

The computation of $[i,j][k,l]F$ is completely same, except that we shall use $v$ for the first slack variable when applying $[k,l]$ and
use $u$ for the second slack variable.
We obtain
$$[i,j][k,l]F=\frac{(-D_v)^{m_2-1}}{(m_2-1)!}\frac{(-D_u)^{m_1-1}}{(m_1-1)!}  \frac{E(ux_j,x_j,vx_l,x_l)}{uv}\Big|_{v=u=1}.$$
This completes the proof.
\end{proof}

By Lemma~\ref{lem-comm}, we can obtain the next corollary. Then we can identify
a monomial operator of $\CTA$ as its corresponding ordered digraph.
\begin{cor}\label{cor-identify}
If two monomial operators $L,L'\in \CTA$ have the same ordered digraph, i.e.,
$D(L)=D(L')$, then $L=L'$.
\end{cor}
\begin{proof}
We prove by induction on $s=\deg(L)$. The base cases $s\leq 1$ is trivial.
Assume the corollary holds for $s-1$ and less. We prove it holds for $s$.

If $D(L)$ has at least two components, say $T_1,T_2,\dots, T_k$ ($k\geq 2$).
For a given $1\leq i\leq k$, if there exist two monomial operators $L_i,L_i'\in \CTA$ such that $D(L_i)=T_i=D(L_i')$, then $L_i=L_i'$ by the induction hypothesis.
The corollary follows by the commutativity of the $L_i$'s using Lemma~\ref{lem-comm}.

If $D(L)$ has only one component, then by the part (2) of Corollary~\ref{cor-basic},
$D(L)$ must be an ordered directed rooted tree with root $r$.
Assume the vertices $r_1,\dots, r_k$ directing to $r$ are arranged from right to left.
For $i=1,\dots,k$, let $T_i$ denote the ordered subtree rooted at $r_i$.
See {Figure~\ref{fig-Ti-cor}} below.
If there exist two monomial operators $L_i,L_i'\in \CTA$ such that
$D(L_i)=T_i=D(L_i')$ for a given $1\leq i\leq k$, then $L_i=L_i'$ by the induction hypothesis. By the Commutativity Rule (Lemma~\ref{lem-comm}), we have
\[
L=[r_k,r]L_k[r_{k-1},r]L_{k-1}\cdots [r_1,r]L_1=[r_k,r]\cdots [r_1,r]L_k\cdots L_1=L'.
\]
\begin{figure}[htpb]
  \centering
  \includegraphics[width=6 cm]{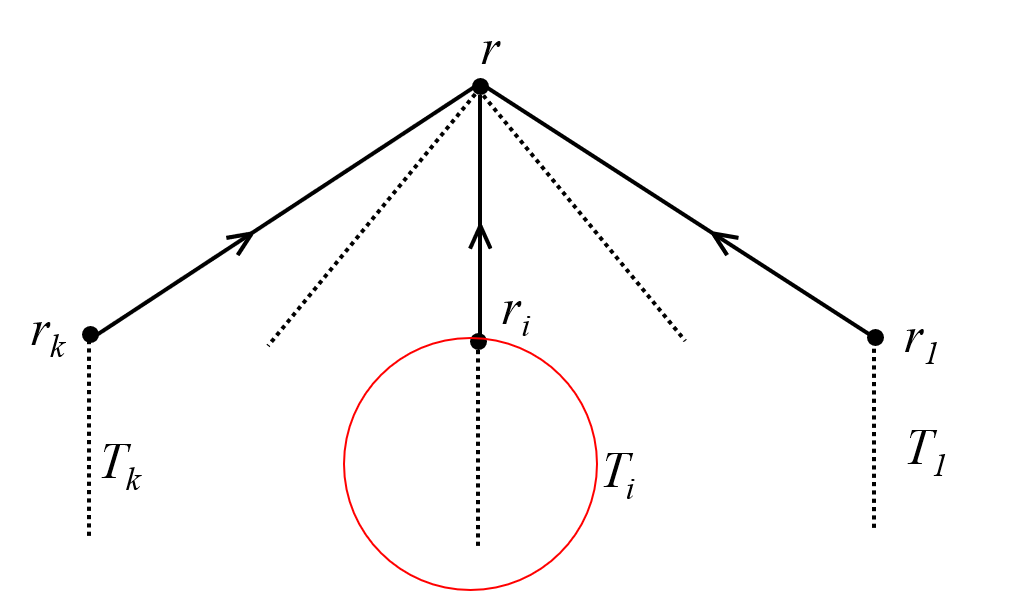}\\
  \caption{$T_i$: the ordered subtree rooted at $r_i$.}\label{fig-Ti-cor}
\end{figure}
This completes the proof.
\end{proof}

The second basic rule in $\CTA$ is the next Exchange Rule.
\begin{lem}[Exchange Rule]\label{lem-Exchangerule}
For distinct $i,j,k\in \{1,2,\dots,n\}$ and $F\in \A_n$,
\[
[j,k] [i,j]=-[i,k][j,i], \quad i.e., \quad
\CT_{x_j=x_k}\CT_{x_i=x_j}F=-\CT_{x_i=x_k}\CT_{x_j=x_i}F.
\]
\end{lem}
\begin{proof}
%For positive integers $m_1,m_2,m_3$ and a rational function $Q(u,v,w)$, let $\overline{D}_m$ be the operator defined by
%\begin{equation}\label{e-struc-7}
%\overline{D}_m \big(Q(u,v,w)\big)=\bigg(\frac{(-D_u)^{m_1-1}}{(m_1-1)!} \frac{(-D_v)^{m_2-1}}{(m_2-1)!}
%\frac{(-D_w)^{m_3-1}}{(m_3-1)!} Q(u,v,w) \bigg)\bigg|_{u=v=w=1},
%\end{equation}
%where $m=(m_1,m_2,m_3)$.
Assume $F\in \A_n$ can be written as
\begin{equation}\label{e-struc-8}
F=\frac{E(x_i,x_j,x_k)}{(1-x_i/x_j)^{m_1}(1-x_j/x_k)^{m_2}(1-x_i/x_k)^{m_3}} ,
\end{equation}
where $x_i=x_j$, $x_j=x_k$ and $x_i=x_k$ are not poles of $E$.

By Corollary \ref{cor-abc}, we can write
\begin{equation}\label{e-struc-8-0}
  F=F_1+F_2=\frac{E_1(x_i,x_j,x_k)}{(1-x_i/x_j)^{m_1}(1-x_j/x_k)^{m'_2}} + \frac{E_2(x_i,x_j,x_k)}{(1-x_j/x_k)^{m'_2}(1-x_i/x_k)^{m_3}},
\end{equation}
where $E_1, E_2$ has no poles at $x_i=x_j$, $x_j=x_k$ and $x_i=x_k$.
Since $x_i=x_j$ is not a pole of $F_2$, we have $[i,j] F_2 = [j,i] F_2=0$ by the part (1) of Lemma~\ref{lem-struc-1}.
Consequently $[j,k][i,j] F_2 =-[i,k][j,i] F_2 =0$.

The lemma then holds by showing that $[j,k][i,j] F_1 =-[i,k][j,i] F_1$.
Applying the constant term operator $[i,j]$ to $F_1$ and using \eqref{e-struc-6-0} gives
\begin{equation}\label{e-F1-ij}
[i,j] F_1= \frac{(-D_u)^{m_1-1}}{(m_1-1)!} \frac{E_1(ux_j,x_j,x_k)}{u(1-x_j/x_k)^{m'_2}} \bigg|_{u=1} =\frac{1}{(1-x_j/x_k)^{m'_2}} \cdot \frac{(-D_u)^{m_1-1}}{(m_1-1)!} \frac{E_1(ux_j,x_j,x_k)}{u} \bigg|_{u=1}.
\end{equation}
Since $x_j=x_k$ is not a pole of $E_1(ux_j,x_j,x_k)$,
applying $[j,k]$ and using \eqref{e-struc-6-0} again give
\begin{align}\label{e-F1-ijjk}
[j,k][i,j]F_1&= \frac{(-D_v)^{m'_2-1}}{(m'_2-1)!}\frac{1}{v} \bigg(\frac{(-D_u)^{m_1-1}}{(m_1-1)!} \frac{E_1(uvx_k,vx_k,x_k)}{u} \bigg|_{u=1}\bigg) \bigg|_{v=1} \notag \\
&= \frac{(-D_u)^{m_1-1}}{(m_1-1)!} \frac{(-D_v)^{m'_2-1}}{(m'_2-1)!} \frac{E_1(uvx_k,vx_k,x_k)}{uv} \bigg|_{u=v=1}.
\end{align}

The situation of computing $[i,k][j,i]F_1$ is a little different. Applying $[j,i]$ to $F_1$
gives
%and using \eqref{e-(2)} in Lemma~\ref{lem-exchange}, then \eqref{e-(1)} and \eqref{e-struc-6-0} give
\begin{align}\label{e-F1-ji}
[j,i] F_1 &= [j,i] \left(\frac{(-D_u)^{m_1-1}}{(m_1-1)!} \frac{E_1(x_i,x_j,x_k)}{(u-x_i/x_j)(1-x_j/x_k)^{m'_2}} \Big|_{u=1}\right) \notag \\
&= \left(\CT_{x_j=x_i^*}\frac{(-D_u)^{m_1-1}}{(m_1-1)!} \frac{E_1(x_i,x_j,x_k)}{(1-ux_j/x_i)(-x_i/x_j)(1-x_j/x_k)^{m'_2}}\right) \bigg|_{u=1}\notag & \text{by} \ \ \eqref{e-(2)}\\
&= \left(\frac{(-D_u)^{m_1-1}}{(m_1-1)!} \CT_{x_j=x_i^*} \frac{E_1(x_i,x_j,x_k)}{(1-ux_j/x_i)(-x_i/x_j)(1-x_j/x_k)^{m'_2}}\right) \bigg|_{u=1}\notag & \text{by} \ \ \eqref{e-(1)}\\
&= \frac{(-D_u)^{m_1-1}}{(m_1-1)!}\frac{E_1(x_i,x_i/u,x_k)}{-u\big(1-x_i/(u x_k)\big)^{m'_2}} \bigg|_{u=1}. & \text{by} \ \ \eqref{e-struc-6-0}
\end{align}
The factor $\big(1-x_i/(u x_k)\big)^{m'_2}$ in the denominator will contribute a pole $x_i=x_k$ after setting $u=1$.

By applying $[i,k]$ to both sides of \eqref{e-F1-ji}, we obtain
\begin{align*}
[i,k][j,i] F_1 &= [i,k]\left( \frac{(-D_u)^{m_1-1}}{(m_1-1)!}\frac{E_1(x_i,x_i/u,x_k)}{-u\big(1-x_i/(u x_k)\big)^{m'_2}} \Big|_{u=1}\right) \\
&= \left(\CT_{x_i=x_k^*}  \frac{(-D_u)^{m_1-1}}{(m_1-1)!}\frac{E_1(x_i,x_i/u,x_k)}{-u\big(1-x_i/(u x_k)\big)^{m'_2}} \right) \bigg|_{u=1}.
\end{align*}
Here the last equation holds by \eqref{e-(2)}.
Using \eqref{e-(1)} and \eqref{e-struc-6-0} again, we obtain
\begin{align}\label{e-F1-jiik}
[i,k] [j,i] F_1%&= \left([i,k]\frac{(-D_u)^{m_1-1}}{(m_1-1)!}\frac{E_1(x_i,x_i/u,x_k)}{-u(1-x_i/u x_k)^{m'_2}}\right) \Big|_{u=1} \notag\\
&= \left(\frac{(-D_u)^{m_1-1}}{(m_1-1)!}\CT_{x_i=x_k^*}\frac{E_1(x_i,x_i/u,x_k)}{-u\big(1-x_i/(u x_k)\big)^{m'_2}}\right) \bigg|_{u=1} \notag \\
&= \left( \frac{(-D_u)^{m_1-1}}{(m_1-1)!}\left(\frac{(-D_v)^{m'_2-1}}{(m'_2-1)!} \frac{E_1(uvx_k,vx_k,x_k)}{-uv} \Big|_{v=1}\right)\right) \bigg|_{u=1} \notag \\
&= \frac{(-D_u)^{m_1-1}}{(m_1-1)!}\frac{(-D_v)^{m'_2-1}}{(m'_2-1)!} \frac{E_1(uvx_k,vx_k,x_k)}{-uv} \bigg|_{u=v=1}.
\end{align}
Comparing \eqref{e-F1-ijjk} with \eqref{e-F1-jiik} yields $[i,k][j,i]F_1 = -[j,k][i,j]F_1$, as desired.
\end{proof}

The third rule is the next V-Formula.
\begin{prop}[V-Formula]
For distinct $i,j,k\in \{1,2,\dots,n\}$ and $F\in \A_n$,
\[
[j,k][i,k] = [i,k][j,k] + [i,k][j,i], \quad i.e., \quad
\CT_{x_j=x_k}\CT_{x_i=x_k}F=\CT_{x_i=x_k}\CT_{x_j=x_k}F +\CT_{x_i=x_k}\CT_{x_j=x_i}F.
\]
\end{prop}
\begin{proof}
Let $F, F_1, F_2$ be as in \eqref{e-struc-8} and \eqref{e-struc-8-0}. We can easily obtain $[j,k][i,k]F_1 =0$ and $[i,k][j,i]F_2 = 0$ by the part (1) of Lemma~\ref{lem-struc-1}. Meanwhile,
we have figured out that
\[
[i,k][j,i] F_1=\frac{(-D_u)^{m_1-1}}{(m_1-1)!}\frac{(-D_v)^{m'_2-1}}{(m'_2-1)!} \frac{E_1(uvx_k,vx_k,x_k)}{-uv} \Big|_{u=v=1}
\]
in the proof of Lemma~\ref{lem-Exchangerule}.
To prove the proposition, we compute $[i,k][j,k]F_1, [j,k][i,k]F_2$ and $[i,k][j,k]F_2$. For the sake of uniformity, we use $u$ as the slack variable when applying $[i,k]$ and use $v$ when applying $[j,k]$.

To get $[i,k][j,k]F_1$, by applying $[j,k]$ to $F_1$ and using \eqref{e-struc-6-0}, we obtain
\begin{equation}\label{e-struc-12}
[j,k]F_1= \frac{(-D_v)^{m'_2-1}}{(m'_2-1)!} \frac{E_1(x_i,vx_k,x_k)}{\big(1-x_i/(vx_k)\big)^{m_1}v} \Big|_{v=1}.
\end{equation}
It is not hard to find that $x_i=x_k$ is a pole of the right-hand side of \eqref{e-struc-12}
after setting $v=1$.
By applying $[i,k]$ to both sides of \eqref{e-struc-12}, we have
\begin{align}\label{e-F1-jkik}
 [i,k][j,k]F_1&= [i,k]\left(\frac{(-D_v)^{m'_2-1}}{(m'_2-1)!} \frac{E_1(x_i,vx_k,x_k)}{\big(1-x_i/(vx_k)\big)^{m_1}v}\right) \bigg|_{v=1} \notag \\
 &= \left(\CT_{x_i=x_k^*}\frac{(-D_v)^{m'_2-1}}{(m'_2-1)!} \frac{E_1(x_i,vx_k,x_k)}{\big(1-x_i/(vx_k)\big)^{m_1}v}\right) \bigg|_{v=1} \notag &\text{by} \ \ \eqref{e-(2)}\\
 &= \left(\frac{(-D_v)^{m'_2-1}}{(m'_2-1)!} \CT_{x_i=x_k^*}\frac{E_1(x_i,vx_k,x_k)}{\big(1-x_i/(vx_k)\big)^{m_1}v}\right) \bigg|_{v=1} \notag
 &\text{by} \ \ \eqref{e-(1)}\\
 &= \left(\frac{(-D_v)^{m'_2-1}}{(m'_2-1)!} \left(\frac{(-D_u)^{m_1-1}}{(m_1-1)!} \frac{E_1(uvx_k,vx_k,x_k)}{uv} \Big|_{u=1} \right) \right) \bigg|_{v=1} \notag &\text{by} \ \ \eqref{e-struc-6-0}\\
 &= \frac{(-D_u)^{m_1-1}}{(m_1-1)!}\frac{(-D_v)^{m'_2-1}}{(m'_2-1)!} \frac{E_1(uvx_k,vx_k,x_k)}{uv} \bigg|_{u=v=1}.
\end{align}
To obtain $[j,k][i,k]F_2$, by applying $[i,k]$ to $F_2$ and using \eqref{e-struc-6-0},
\[
[i,k]F_2= \frac{(-D_u)^{m_3-1}}{(m_3-1)!} \frac{E_2(ux_k,x_j,x_k)}{(1-x_j/x_k)^{m'_2}u} \bigg|_{u=1}= \frac{1}{(1-x_j/x_k)^{m'_2}}\cdot \frac{(-D_u)^{m_3-1}}{(m_3-1)!} \frac{E_2(ux_k,x_j,x_k)}{u} \bigg|_{u=1}.
\]
Since $x_j=x_k$ is not a pole of $\frac{(-D_u)^{m_3-1}}{(m_3-1)!} \frac{E_2(ux_k,x_j,x_k)}{u} \big|_{u=1}$, we can apply $[j,k]$ to the above equation and use \eqref{e-struc-6-0}.
This gives
\begin{align}\label{e-F2-ikjk}
 [j,k][i,k]F_2&= \frac{(-D_v)^{m'_2-1}}{(m'_2-1)!}\frac{1}{v}\bigg( \frac{(-D_u)^{m_3-1}}{(m_3-1)!} \frac{E_2(ux_k,vx_k,x_k)}{u} \Big|_{u=1} \bigg) \bigg|_{v=1} \notag \\
 &= \frac{(-D_v)^{m'_2-1}}{(m'_2-1)!}\frac{(-D_u)^{m_3-1}}{(m_3-1)!} \frac{E_2(ux_k,vx_k,x_k)}{uv} \bigg|_{v=u=1}.
\end{align}

Similar to $[j,k][i,k]F_2$, we can obtain
\begin{equation}\label{e-F2-jkik}
[i,k][j,k]F_2 = \frac{(-D_v)^{m'_2-1}}{(m'_2-1)!}\frac{(-D_u)^{m_3-1}}{(m_3-1)!} \frac{E_2(ux_k,vx_k,x_k)}{uv} \Big|_{v=u=1}.
\end{equation}

By \eqref{e-F1-jiik} and \eqref{e-F1-jkik}, we have $[i,k][j,k]F_1 + [i,k][j,i]F_1=0$.
Then
\[
[j,k][i,k]F_1 = [i,k][j,k]F_1 + [i,k][j,i]F_1
\]
since both sides vanish.
By \eqref{e-F2-ikjk} and \eqref{e-F2-jkik}, we have $[j,k][i,k]F_2 = [i,k][j,k]F_2$.
Then
\[
[j,k][i,k]F_2=[i,k][j,k]F_2+[i,k][j,i]F_2
\]
since $[i,k][j,i]F_2=0$.
Thus, the V-Formula holds.
\end{proof}

From the Exchange Rule and the V-Formula, we can obtain the next general exchange property of the operators in $\CTA$.
\begin{lem}\label{lem-genexchange}
For any monomial operator $L\in \CTA$ and distinct $i,j,k\in \{1,2,\dots,n\}$,
\begin{equation}\label{e-struc-11}
[j,k]L[i,j]=-[i,k](L|_{j=i})[j,i].
\end{equation}
\end{lem}
\begin{proof}
The operator on the right-hand side (without the sign) is indeed obtained from that on the left-hand side by exchanging $i$ and $j$. Thus their corresponding digraphs are either both forests or both not forests. So we may assume that they are nonzero.

In general, assume
\[
L=L_{s+1}[i_s,j]L_s\cdots [i_1,j]L_1,
\]
where all the appearances of $j$ are displayed. Thus all the $L_h$'s
contain no $j$. Furthermore, no $i_h$ appears to the left of $[i_h,j]$, otherwise
this would make $L$ to be zero by the part (2) of Lemma~\ref{lem-struc-1}.
It follows that
\[
[j,k]L[i,j]=[j,k]L_{s+1}[i_s,j] L_s \cdots [i_1,j]L_1[i,j] = [j,k][i_s,j]\cdots [i_1,j][i,j]  L_{s+1} \cdots L_1
\]
by the Commutativity Rule (Lemma~\ref{lem-comm}).
Since the left-hand side of \eqref{e-struc-11} is nonzero, the $L_h$'s contain no $i$ also. Then by the same reasoning, we have
\[
[i,k](L|_{j=i})[j,i]=[i,k]L_{s+1}[i_s,i] L_s \cdots [i_1,i]L_1[j,i] = [i,k][i_s,i]\cdots [i_1,i][j,i]L_{s+1}\cdots L_1.
\]
We complete the proof by showing that
\begin{equation}\label{e-struc-10}
[j,k][i_s,j]\cdots [i_1,j][i,j]=-[i,k][i_s,i]\cdots [i_1,i][j,i].
\end{equation}

We prove \eqref{e-struc-10} by induction on $s$. The $s=0$ case is the Exchange Rule (Lemma~\ref{lem-Exchangerule}). Assume \eqref{e-struc-10} holds for $s-1$.
Denote the left-hand side of \eqref{e-struc-10} by $H$.
Using the V-Formula, we have
\[
[i_1,j][i,j]=[i,j][i_1,j]+[i,j][i_1,i].
\]
Substituting this into $H$ yields
\[
H=[j,k][i_s,j]\cdots [i_2,j][i,j][i_1,j]+[j,k][i_s,j]\cdots [i_2,j][i,j][i_1,i].
\]
By the induction hypothesis, we have
\[
[j,k][i_s,j]\cdots [i_2,j][i,j]=-[i,k][i_s,i]\cdots [i_2,i][j,i].
\]
Then
\[
H=-[i,k][i_s,i]\cdots [i_2,i]([j,i][i_1,j]+[j,i][i_1,i]).
\]
Using the V-Formula again gives
\[
H=-[i,k][i_s,i]\cdots [i_2,i][i_1,i][j,i],
\]
which is the right-hand side of \eqref{e-struc-10}.
\end{proof}

\subsection{The spanning elements of $\CTA$}

In this subsection, we find the spanning elements of $\CTA$.

Given an ordered digraph $D$, we call a monomial operator $L$ with $D(L)=D$  the realization of $D$ ({denoted by $L(D)$}).
By the part (2) of Corollary~\ref{cor-basic} and the Commutativity Rule, to obtain the realization of a given ordered digraph $D$, it suffices to assume
that $D$ is an ordered directed rooted tree $T$.
For a given $T$, we read out {$L(T)$} according to the direct edges of $T$ level by level from up to down and left to right. That is, we first read out the highest and left-most edge $[i_1,j_1]$, this is the left-most commutator in {$L(T)$}. Then, read out the next edge $[i_2,j_2]$ to the right and in the same level of $[i_1,j_1]$. Continue this until the right-most edge $[i_{s_1},j_{s_1}]$ in the highest level. Then read out the edges in the next level from left to right, say $[i_{s_1+1},j_{s_1+1}]\cdots [i_{s_1+s_2},j_{s_1+s_2}]$.
The reading ends at the lowest level and right-most edge $[i_s,j_s]$ of $T$.
Finally,
\[
{L(T)}=[i_1,j_1]\cdots[i_{s_1},j_{s_1}][i_{s_1+1},j_{s_1+1}]\cdots [i_{s_1+s_2},j_{s_1+s_2}] \cdots [i_s,j_s].
\]
For example, the realization of {Figure~\ref{fig-DL-ex}} is: $[6,2][1,2][7,2][4,2][5,6][3,7][8,7]$.
Having Corollary~\ref{cor-identify} and the discussion above, we can identify a monomial operator $L\in \CTA$ as an ordered directed rooted forest $D$, and vise versa.
In what follows, we always assume that a forest is ordered, directed and rooted, unless specified otherwise.

Note that the above way of reading is not unique. From the proof of Corollary~\ref{cor-identify}, we can give another way of reading according to subtrees (we omit the details).
But the realizations obtained by these two distinct ways of reading are equivalent by the Commutativity Rule.

\def\S{\mathbf{S}}
\def\r{\mathbf{r}}
\def\Par{\mathrm{Par}}
\def\compat{\;\; \mathrm{compat} \;\;}

To state our structural result, we need more notations.

Let $\S=\{S_1,S_2,\dots,S_\ell\}$ be a partition of $N=\{1,2,\dots,n\}$ with $\ell=|\S|$ parts. That is, $N$ is the disjoint union of the nonempty sets $S_i$ for $i=1,\dots, \ell$. For convenience, we arrange the $S_i$ increasingly according to their minimal elements. Hence in particular $1\in S_1$.
Denote by $\Par(N,\ell)$ the set of all partitions of $N$ with $\ell$ blocks, and let $\Par(N)=\cup_{\ell=1}^n \Par(N,\ell) $.
A tree $T$ is said to be on $S$ if the vertex set of $T$ is $S$.
A forest $D$ is said to be on $\S$ if its connected components are trees on $S_1$, \dots, $S_{\ell}$. It is clear that $D$ has exactly $n-|\S|$ edges. Note that $|S_i|=1$ gives the empty tree with the single-vertex set $S_i$.
Denote by $\mathcal{D}_{\S}$ the set of all rooted forests on $\S$. Let $\r=(r_1,\dots, r_\ell)$ be a sequence of integers. The sequence $\r$ is \emph{compatible} with $\S$ if $r_i\in S_i$ for every $i$. We use $\r \compat \S$ to represent this.
Let $\mathcal{D}_{\S,\r}$ be the set of forests on $\S$ with given roots $\r$. Denote by $\mathcal{T}_{S}$ the set of all rooted trees on $S\subseteq N$,
and let $\mathcal{T}_{S,r}$ be the set of all rooted trees on $S$ with given root $r$.
We remark that the $\S$ in $\mathcal{D}_{\S,\r}$ is a partition of $N$. Hence, if $D\in \mathcal{D}_{\S,\r}$, then the vertex set $V(D)=N$.
But for $T\in \mathcal{T}_{S}$, $V(T)=S$ may not be $N$. Isolated vertices are allowed in $D$, but not in $T$ (except for single-vertex graphs). This inconsistency does not lead to misunderstanding in this paper.

\def\Dinc{\mathcal{D}^{\textrm{inc.}}}
\def\Dninc{\mathcal{D}^{\textrm{n.inc.}}}
\def\Tinc{\mathcal{T}^{\textrm{inc.}}}
\def\Tninc{\mathcal{T}^{\textrm{n.inc.}}}

\def\Daug{\overline{\mathcal{D}}^{\textrm{inc.}}}
\def\Daugn{\overline{\mathcal{D}}^{\textrm{n.inc.}}}
\def\Taug{\overline{\mathcal{T}}^{\textrm{inc.}}}
\def\Taugn{\overline{\mathcal{T}}^{\textrm{n.inc.}}}

A directed edge $i\to j$ is \emph{increasing} if $i>j$.
A forest $D$ is \emph{increasing} if its directed edges are all increasing.
It is \emph{nearly increasing} if all the edges not end at the root are increasing. Let the superscript $^{\textrm{inc.}}$ and $^{\textrm{n.inc.}}$ stand for increasing and
nearly increasing respectively. For instance, $\Dninc_{\S,\r}$ is the set of all nearly increasing forests on $\S$ with given roots $\r$.
For a given integer $0\leq s\leq n-1$, we write
\[
\Dinc_{s}:=\bigcup_{\S \in \Par(N,n-s)}\Dinc_{\S}
\quad \text{ and } \quad
\Dninc_{s}:=\bigcup_{\S \in \Par(N,n-s)}\Dninc_{\S},
\]
where
\[
\Dinc_{\S}:=\bigcup_{\r \compat \S} \Dinc_{\S,\r}
\quad \text{ and } \quad
\Dninc_{\S}:=\bigcup_{\r \compat \S} \Dninc_{\S,\r}.
\]
If furthermore for each inner vertex $v$, the vertices directing to $v$ have increasing labels
from left to right, then we say $D$ is \emph{augmented increasing}. Similarly, we can define an \emph{augmented nearly increasing} forest.
Define $\Daugn_{\S,\r}$ to be the set of
all augmented nearly increasing forests on $\S$ with given roots $\r$, and similarly for $\Daug_{\S,\r}$. Then similarly we define
\[
\Daug_{s}:=\bigcup_{\S \in \Par(N,n-s)}\Daug_{\S} \quad \text{ and } \quad
\Daugn_{s}:= \bigcup_{\S \in \Par(N,n-s)}\Daugn_{\S},
\]
where
\[
\Daug_{\S}:=\bigcup_{\r \compat \S} \Daug_{\S,\r} \quad \text{ and } \quad
\Daugn_{\S}:= \bigcup_{\r \compat \S} \Daugn_{\S,\r}.
\]
In particular, $\Daug_{n-1}=\Daug_{\{N\}}$ and $\Daugn_{n-1}=\Daugn_{\{N\}}$.
We can change
forests $\mathcal{D}$ to trees
$\mathcal{T}$ in the above notations.
For instance,
$\Taug_{S}$ is the set of all augmented increasing (ordered directed rooted) trees on $S$ (hence has exactly $|S|-1$ edges).

The next result shows that any monomial operator $L$ realizing a forest can be converted to be (nearly) increasing.
\begin{prop}\label{prop-nearincreasing}
Let $\S$ be a partition of $N$. Then for any $D\in \mathcal{D}_{\S}$, there exists a $D'\in \Dninc_{\S}$ such that $L(D')=\pm L(D)$.
Furthermore, if $\S=\{N\}$, then for any tree $T\in \mathcal{T}_{N}$,
there exists an increasing tree $T^*\in \Tinc_{N}$ such that $L(T^*)=\pm L(T)$.
\end{prop}
There may be a confusion in this proposition for the first-time reader. It is possible that different forests $D$ and $D'$
have equivalent realizations $L(D)=L(D')$ in $\CTA$. For example, $[2,4][3,2][1,3]=[1,4][2,1][3,1]$, see their corresponding digraphs below (Figure~\ref{fig-dig-eq-L}).
The equality can be obtained using the Exchange Rule twice. The digraph of the operator $[1,4][2,1][3,1]$ is nearly increasing.
\begin{figure}[htpb]
  \centering
  \includegraphics[width=4 cm]{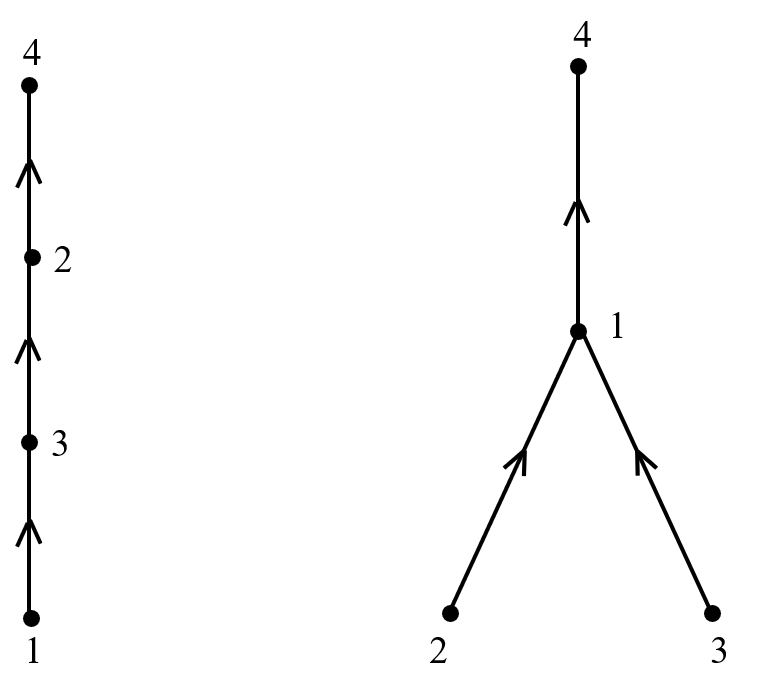}\\
  \caption{The digraphs of $[2,4][3,2][1,3]$ and $[1,4][2,1][3,1]$.}\label{fig-dig-eq-L}
\end{figure}

\begin{proof}
To prove the first part, we begin with the proof of the case when $D$ is a tree on $S$ with a root $r$.
Assume the realization of $D$ is
\[
L=L(D)=[i_{s+l},j_{s+l}]\cdots [i_{s+1},j_{s+1}][i_s,j_s]\cdots [i_1,j_1],
\]
where $i_{s+1},\dots,i_{s+l}$ are all the vertices directing to the root $r$ (so $j_{s+1}=\cdots=j_{s+l}=r$).
Define $\nu=\nu(L)$ to be the integer such that
$i_k>j_k$ for $k=1,2,\dots, \nu-1$ and $i_\nu<j_\nu$.
If $\nu\ge s+1$ then $L$ is nearly increasing.
We show that if $\nu\le s$ then we can find a monomial $L'$ such that
$L'=-L$ and $\nu(L')\ge \nu(L)+1$.

Since $j_{\nu}$ is not the root, there exists
an $h>\nu$ such that $[i_h,j_h]=[j_\nu, j_h]$ appears in $L$ (It is better to understand this by imaging the corresponding tree).
Notice that $j_h$ may be the root $r$ if $h>s$.
Denote by $\overline{L}= \cdots [i_\nu,j_\nu]$ the ``subword'' of $L$ up to $[i_\nu,j_\nu]$.
By Lemma~\ref{lem-genexchange},
there exists $\overline{L}'[j_\nu,i_\nu]=-\overline{L}$. Thus,
let $L'$ be obtained from $L$ by replacing $\overline{L}$ by
$\overline{L}'[j_\nu,i_\nu]$.
This operator $L'$ satisfies $L'=-L$ and $\nu(L')\ge \nu(L)+1$.
Such operation can continue until all the $[i_k,j_k]$ for $1\leq k\leq s$ are increasing,
i.e., all the $i_k>j_k$. Denote this operator by $L_1$. Then $D'=D(L_1)$ is nearly increasing and satisfies $L_1=\pm L$.
The statement that $D'$ is also on $S$ follows by the observation that all the operations we performed do not change the vertex set.

Suppose $D$ has several components, say $T_1,\dots, T_k$. Then by the Commutativity Rule
$L(D)=L(T_1)L(T_2)\cdots L(T_k)$, where $L(T_i)$ commutes with $L(T_j)$ for all $i,j$.
We have shown that for each $T_i$, there exists nearly increasing $T_i'$ such that $L(T_i')=\pm L(T_i)$.
The forest consisting of $T_1', T_2',\dots, T_k'$ is the desired $D'$.

To prove the second part, we use the result in the first part to obtain a nearly increasing $T'$ on $N$ with root $r$ such that $L(T')=\pm L(T)$.
By the Commutativity Rule, we may assume the realization of $T'$ is given by
\[
L_1=L(T')=[u_l,r][u_{l-1},r]\cdots [u_1,r][i_s,j_s]\cdots [i_1,j_1],
\]
where $i_m>j_m$ for $m=1,\dots,s$ and $s+l=n-1$.
Define $\mu=\mu(L_1)$ to be the integer such that $u_m>r$ for $m=1,\dots,\mu-1$,
but $u_{\mu}<r$. By Lemma~\ref{lem-extrachange}, we find that
\[
L_2=[r,u_l][u_{l-1},r]\cdots [u_{\mu},r][u_{\mu-1},r]\cdots [u_1,r][i_s,j_s]\cdots [i_1,j_1]=-L_1.
\]
Note that if $L_1\notin \CTA^{n-1}$ then we can not use Lemma~\ref{lem-extrachange} and the above equation does not hold in general.
Denote the subword $[r,u_l][u_{l-1},r]\cdots [u_{\mu},r]$ in $L_2$ by $L''$. By Lemma~\ref{lem-genexchange},
we find
\[
\overline{L}''=[u_{\mu},u_l][u_{l-1},u_{\mu}]\cdots [u_{\mu+1},u_{\mu}][r,u_{\mu}]=-L''.
\]
Let $L_3$ be obtained from $L_2$ by replacing $L''$ by $\overline{L}''$.
Then, $L_3$ satisfies $L_3=L_1$ and $\mu(L_3)\geq \mu(L_1)+1$.
Such operation can continue until we find an $L^*$ such that $D(L^*)$ is increasing and $L^*=\pm L$.
\end{proof}

We find that the realization of any nearly increasing tree can be written as a linear combination of
the realizations of augmented nearly increasing trees.
\begin{lem}\label{lem-treelinear}
Let $T^{(r)}\in \Tninc_{S,r}$  be a nearly increasing tree on $S$ with the root $r$.
Then, its realization $L(T^{(r)})$ can be written as a linear combination of the elements in
$B=\{L(T): T\in \Taugn_{S,r}\}$.
\end{lem}
\begin{proof}
We prove by induction on $|S|$. The base cases $|S|\le 2$ are trivial. Assume the lemma holds for $|S|<s$. We prove the lemma for the $|S|=s$ case.

Let $u_k,\dots,u_1$ be all the children of the root $r$ from left to right, and $T_{u_i}$ be the subtree of $T^{(r)}$ with the root $u_i$ for $i=1,\dots,k$.
Then we can write
\[
L(T)=[u_k,r]\cdots [u_1,r]L(T_{u_k})\cdots L(T_{u_1})=[u_k,r]L(T_{u_k})\cdots [u_1,r] L(T_{u_1}),
\]
where the last equality can be obtained by the Commutativity Rule or the subtree reading of a digraph.
Note that since $T^{(r)}$ is nearly increasing, every subtree $T_{u_i}$ is increasing and the root $u_i$ is the smallest vertex of $T_{u_i}$. By the induction hypothesis, for each $i\in \{1,\dots,k\}$
\[
L(T_{u_i})=\sum_{T'_j} c_j L(T'_j),
\]
where $c_j\in \mathbb{C}$, $T'_j$ ranges over all augmented nearly increasing trees with the root $u_i$ and the same vertex set of $T_{u_i}$. Since $u_i$ is the smallest vertex of $T_{u_i}$, all these $T'_j$ are in fact augmented increasing.
Thus we may assume
$T_{u_i}$ itself to be augmented increasing.

Next we prove by another induction on the inversion number of $T^{(r)}$ at the root $r$, defined by $\inv(T^{(r)};r)=\inv(u_ku_{k-1}\cdots u_1)$.
If $\inv(T^{(r)};r)=0$, then $T^{(r)}$ is augmented nearly increasing and $L(T^{(r)})$ is itself an element of $B$.
Assume the lemma holds
for $\inv(T^{(r)};r)<t$. We prove the lemma for $\inv(T^{(r)};r)=t$.

Let $b=u_i>u_{i-1}=a$ be the right-most adjacent decent of the sequence $u_ku_{k-1}\cdots u_1$.
That is, there does not exist a $j<i$ such that $u_j>u_{j-1}$.
Then we can write
\[
L(T^{(r)})=L_4 [b,r] L_3 [a,r] L_2 L_1,
\]
where $L_2=L(T_a)$ and $L_3=L(T_b)$ are the realizations of the subtrees of $T^{(r)}$ with roots $a$ and $b$ respectively. Notice that $L_3$ is free of $r$ and $a$, and hence commutes with $[a,r]$ by the Commutativity Rule.
Together with the V-Formula, we have
\[
L(T^{(r)})=L_4[b,r][a,r]L_3L_2L_1
=L_4[a,r][b,r]L_3L_2L_1+L_4[a,r][b,a]L_3L_2L_1.
\]
For a digraph $G$, if $L(G)$ can be written as a linear combination of the elements in $B$, we say that $G$ is \emph{representable}.
Denote by $L':=L_4[a,r][b,r]L_3L_2L_1$ and $L'':=L_4[a,r][b,a]L_3L_2L_1$. We will show that
$L'$ and $L''$ are both representable.

Since $L_2$ commutes with $L_3$, and $L_2=L(T_a)$ commutes with $[b,r]$, we can write
\[
L'=L_4 [a,r]L_2[b,r]L_3L_1.
\]
Thus, its corresponding digraph $D(L')$ is just obtained from $T^{(r)}$ by exchanging $T_a$ and $T_b$
(See Figure~\ref{fig-to-aug}), and hence has inversion number $t-1$. By the induction hypothesis on $\inv(T^{(r)};r)$,
$L'$ is representable.

For $L''$, we can see that
its corresponding digraph $D'':=D(L'')$ is obtained from $T^{(r)}$ by removing $b\to r$ and adding
$b\to a$ as the left-most child of $a$, with $T_b$ attached to $b$ (See Figure~\ref{fig-to-aug}). Since $T_b$ and $T_a$ are increasing,
the subtree $T_a(L'')$ of $D(L'')$ with the root $a$ is increasing, but need not be augmented increasing. By the induction hypothesis on $|S|$,
the realization of $T_a(L'')$ can be written as
\[
L(T_a(L''))= [b,a]L_3L_2 = \sum_T c(T) L(T),
\]
where $T$ ranges over all augmented nearly increasing rooted trees with the root $a$ and the same vertex set of $T_a(L'')$.
Since $a$ is the smallest vertex of $T_a(L'')$, these $T$'s are in fact augmented increasing.
Thus we may assume $T_a(L'')$ itself to be augmented increasing. Now $D''$ has a smaller $\inv(D'';r)$. By the induction hypothesis on $\inv(T^{(r)};r)$, $L''$ is representable. This completes the proof.
\begin{figure}[htpb]
  \centering
  \includegraphics[width=13 cm]{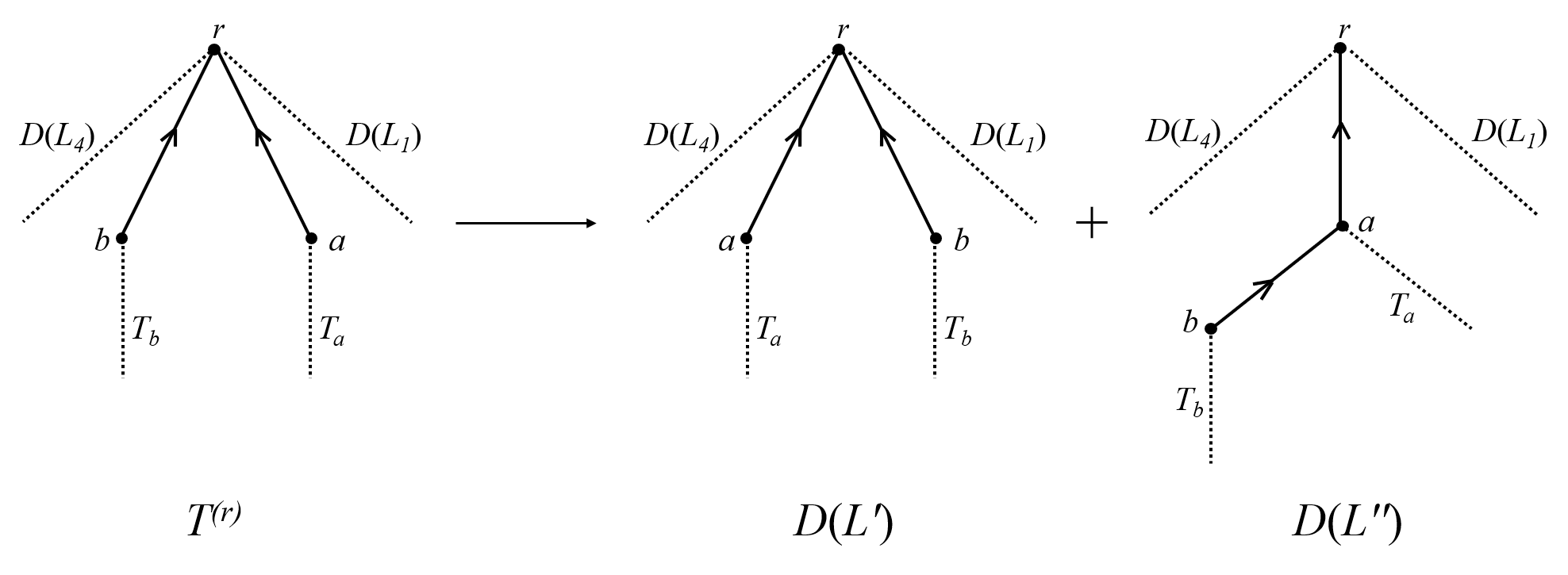}\\
  \caption{The illustration of obtaining $D(L')$ and $D(L'')$ from $T^{(r)}$.}\label{fig-to-aug}
\end{figure}
\end{proof}

By Proposition~\ref{prop-nearincreasing} and Lemma~\ref{lem-treelinear},
we obtain the spanning elements of $\CTA^s$.
\begin{prop}\label{prop-span}
For $1\le s\le n-2$, the subspace $\CTA^{s}$ is spanned by $B^s:=\{L(D): D\in \Daugn_{s}\}.$
The subspace $\CTA^{n-1}$ is spanned by $B^{n-1}:=\{L(D): D\in \Taug_{n-1}\}.$
\end{prop}
\begin{proof}
We show that any monomial operator $L=[i_{s},j_{s}]\cdots [i_1,j_1]$ can be written as a linear combination of the elements of $B^s$.

Let $D=D(L)$ be the corresponding digraph of $L$. By Corollary~\ref{cor-basic} and Proposition~\ref{prop-nearincreasing}, we may assume that $D$ is a nearly increasing
(ordered directed and rooted) forest for $1\le s\le n-2$, otherwise $L=0$.
Suppose $D$ has $k=n-s$ connected components, say $T_{r_1},\dots,T_{r_k}$,
where $r_i$ is the root of the tree $T_{r_i}$ for $i=1,\dots,k$.
Then, we can write
\[
L= L(T_{r_1})\cdots L(T_{r_k}).
\]
Here the $L(T_{r_i})$'s  commute with each other by the Commutativity Rule.
By Lemma~\ref{lem-treelinear}, for each $i$,
\[
L(T_{r_i}) = \sum_{T^{(r_i)}}  c(T^{(r_i)}) L(T^{(r_i)}),
\]
where $T^{(r_i)}$ and $T^{(i)}$ in the next equation range over all augmented nearly increasing trees with root $r_i$ and the same vertex set of $T_{r_i}$ for each $i$. It follows that $L$ can be written in the form
\[
L = \sum_{T^{(1)},\dots,T^{(k)}} c(T^{(1)})\cdots c(T^{(k)}) L(T^{(1)})\cdots L(T^{(k)}).
\]
This shows the first assertion.

For $s=n-1$, we may assume $D=D(L)$ to be increasing by Proposition~\ref{prop-nearincreasing}. The proposition then follows by Lemma~\ref{lem-treelinear} and the similar argument as above.
\end{proof}

\subsection{The basis of $\CTA$}
In the previous subsection, we find the spanning elements of $\CTA^s$ for each $s$.
In this subsection, we show that these spanning elements are linearly independent, and hence indeed form a basis of $\CTA^s$.

We find a simple result that can be used to construct an orthogonality of the monomial operators of $\CTA$.
The following rational function plays an important role in our construction. For an ordered directed rooted forest $D$, define
\begin{equation}\label{e-E-D}
  \epsilon(D)= \prod_{i\to j\in E(D)} \frac{1}{1-x_i/x_j}.
\end{equation}
We have the next basic lemma.
\begin{lem}\label{lem-s=1}
Let $D$ be an ordered directed rooted forest
and $P$ be a Laurent polynomial. Then
\begin{align}\label{e-ij-D}
  \CT_{x_i=x_j} P \epsilon(D)=\left\{
                              \begin{array}{ll}
                               P|_{x_i=x_j} \epsilon(D/{i\to j}), & \mbox{ if } i\to j \in E(D); \\
                -P|_{x_i=x_j}\epsilon(D/{i\to j}), & \mbox{ if } j\to i \in E(D); \\                                0, & \mbox{ otherwise.}
                              \end{array}
                            \right.
\end{align}
\end{lem}
\begin{proof}
If $i\to j$ is an edge of $D$, then $\epsilon(D)=\frac{1}{1-x_i/x_j}\epsilon(D')$, where $D'$ is obtained from $D$ by removing the edge $i\to j$.
By direct computation we have
$$\CT_{x_i=x_j} P\epsilon(D)=\CT_{x_i=x_j}\frac{1}{1-x_i/x_j} P\epsilon(D')=P\epsilon(D')\Big|_{x_i=x_j},$$
which is just $P\big|_{x_i=x_j}\epsilon(D/{i\to j})$.

When $j\to i$ is an edge of $D$, similar computation yields
$$\CT_{x_i=x_j} P\epsilon(D)=\CT_{x_i=x_j}\frac{1}{1-x_j/x_i}P\epsilon(D')
=\CT_{x_i=x_j}\frac{-x_i/x_j}{1-x_i/x_j}P\epsilon(D')
=-P\epsilon(D')\Big|_{x_i=x_j},$$
which is $-P|_{x_i=x_j}\epsilon(D/{i\to j})$.

Finally, if neither $i\to j$ nor $j\to i$ is an edge of $D$, then
$x_i=x_j$ is not a pole of $\epsilon(D)$. Thus $\CT\limits_{x_i=x_j} P\epsilon(D) =0$, as desired.
\end{proof}

Repeatedly use of Lemma~\ref{lem-s=1} gives the next result.
\begin{lem}\label{lem-struct-2}
Let $S,S'$ be subsets of $\{1,\dots,n\}$ and $r\in S\cap S'$.
Suppose $T\in \mathcal{T}_{S,r}$ is a tree on $S$ with the root $r$, $T'\in \mathcal{T}_{S'}$
and $P$ is a Laurent polynomial. Then one of the following cases holds:
\begin{enumerate}
\item $L(T) P \epsilon(T')=0$. A sufficient condition is
    $S \not \subseteq S'$.

\item $L(T) P \epsilon(T')$ is of the form $\pm P' \epsilon(T'')$.
Here $P'$ is a Laurent polynomial obtained from $P$ by substituting $x_i$ by $x_r$ for every $i\in S$, and $T''$ is a tree on $(S'\setminus S)\cup \{r\}$ and is obtained from $T'$ through certain contractions.
In particular, $\epsilon(T'')=1$ if $S=S'$.
\end{enumerate}
\end{lem}
\begin{proof}
Suppose $L=L(T)=[i_s,j_s]\cdots [i_1,j_1]$ is the realization of $T$.
Let $D_1=T'$ and recursively define $D_{k+1}=D_{k}/{i_k\to j_k}$ for $k=1,2,\dots, s$ if $i_k\to j_k$ or $j_k\to i_k$ is an edge of $D_k$. Observe that $D_{k}$ is a tree on $S'\setminus \{i_1,\dots, i_{k-1}\}$, and $D_{k+1}$ is not defined if $i_k\not\in S'$.
By Lemma~\ref{lem-s=1}, if $D_{k+1}$ is not defined for some $k$, this implies $[i_k,j_k]\cdots [i_1,j_1] P \epsilon(T')=0$; otherwise we can obtain $D_{s+1}$
and $S\subseteq S'$. Moreover,
we know that $D_{s+1}$ is a tree on $(S'\setminus S) \cup \{r\}$ and $L P\epsilon(T')$ is of the form $\pm P' \epsilon(T'')$. Here $T''=D_{s+1}$
and $P'$ is obtained from $P$ by successive replacements $x_{i_1}=x_{j_1}$, \dots, $x_{i_s}=x_{j_s}$. Since $T$ is a tree on $S$, the $x_i$ for all $i\in S$ of $P$
are replaced by $x_r$.
\end{proof}

Let $\S_1,\S_2$ be two partitions of $\{1,2,\dots,n\}$. We say that $\S_1$ is a \emph{refinement} of $\S_2$ if each block of $\S_1$ is contained in one of the blocks of $\S_2$.
It is natural to generalize Lemma~\ref{lem-struct-2} to forests. Then, we have the next three consequences.
\begin{cor}\label{cor-ortho-1}
Suppose $P$ is a Laurent polynomial, and $D_1, D_2$ are two ordered directed rooted forests on $\S_1$ and $\S_2$ respectively.
If $L(D_1) P\epsilon(D_2)\neq 0$, then $\S_1$ must be a refinement of $\S_2$. As a consequence, $|E(D_2)|\geq |E(D_1)|$.
Moreover, $L(D_1) P\epsilon(D_2)$ is of the form $\pm P'\epsilon(D_3)$. Here $D_3$ is a forest satisfying
$|E(D_3)|=|E(D_2)|-|E(D_1)|$,
and $P'$ is obtained from $P$ by replacing each $x_i$ by its corresponding $x_r$,
where $r$ is the root of the component of $D_1$ containing $i$.
\end{cor}

%\begin{lem}\label{lem-ortho-1}
%Suppose $D_k$ is an ordered directed rooted forest on $\S_k$ for $k=1,2$, and $P$ is a Laurent polynomial.
%Then $L(D_1) P\epsilon(D_2)$ is either $0$ or $\pm P'\epsilon(D_3)$
%for certain forest $D_3$. In the latter case, i) $\S_1$ must be a refinement of $\S_2$; ii) $|E(D_3)|=|E(D_2)|-|E(D_1)|$;
%iii) $P'$ is obtained from $P$ by replacing each $x_i$ by $x_r$
%where $r$ is the root of the component of $D_1$ containing $i$.

%\begin{align*}
%L(D_1) \epsilon(D_2)& \in \mathbb{C} \qquad \textrm{if } |E(D_1)|=|E(D_2)|, \\
%L(D_1)\big(\epsilon(D_1)\big)&=1.
%\end{align*}
%\end{lem}

\begin{cor}\label{cor-ge}
If $D_1$ and $D_2$ are two forests with $|E(D_2)|< |E(D_1)|$, then $L(D_1) \epsilon(D_2)=0$.
\end{cor}

\begin{cor}\label{cor-eq}
If two forests $D_1$ and $D_2$ have the same number of edges, then $L(D_1) \epsilon(D_2)\in \{-1,0,1\}$. In particular, $L(D_1) \epsilon(D_1)=1$.
\end{cor}

Recall that $B^s=\{L(D): D\in \Daugn_{s}\}$ for $s=1,\dots,n-2$ and
$B^{n-1}=\{L(D): D\in \Taug_{n-1}\}$.
%Recall that $B^s=\{L(D): D\in \overline{\mathcal{D}}^0_{s}\}$ for $s=1,\dots,n-2$ and
%$B^{n-1}=\{L(D): D\in \overline{\mathcal{T}}_{n-1}\}$.
We find that for a fixed $s$ the elements of $B^s$ admit the next kind of orthogonality.
\begin{lem}\label{lem-orthogonality-2}
For a fixed $s\in \{1,\dots,n-1\}$, let $D_i,D_j$ be two distinct forests of $\Daugn_{s}$ (trees of $\Taug_{n-1}$ if $s=n-1$)
with the same roots. Then
$L(D_i): \A_n(D_j) \mapsto \{0\}$. Consequently,
\begin{equation}\label{e-orthogonality-n-2}
L(D_i) \epsilon(D_j)=\delta_{i,j},
\end{equation}
where $\epsilon(D_j)$ is defined as in \eqref{e-E-D} and $\delta_{i,j}$ is the Kronecker delta.
\end{lem}
\begin{proof}
We prove by induction on $n$, the number of variables. The lemma clearly holds for
$n\le 2$. Assume the lemma holds for $n-1$ and less.

Let a leaf $u$ of $D_i$ be the right-most child of an inner vertex $v$. Since $D_i$ is an augmented nearly increasing forest, $u$ is also the largest child of $v$, and $u>v$ if $v$ is not a root. Then we can write $L(D_i)=L(D_i/{u\to v})[u,v]$.

If $u\to v\not\in E(D_j)$, then by the assumption that $D_i$ and $D_j$ have the same roots, $u$ is not a root of $D_j$. We then discuss two cases: i) if $v$ is not a root of $D_i$, then $u>v$ so that $v\to u$ is not increasing. Together with that neither $u$ and $v$ is a root of $D_j$ gives $v\to u\notin E(D_j)$; ii) if $v$ is a root of $D_i$ and $D_j$, then
$v\to u$ cannot be an edge of $D_j$. In both cases,
$v\to u\notin E(D_j)$.
Thus $[u,v]F=0$ for any $F\in \A_n(D_j)$.
It follows that $L(D_i)(F)=L(D_i/{u\to v})[u,v](F)=0$.

On the other hand, assume that $u\to v$ is an edge of $D_j$.
By the part (2) of Lemma~\ref{lem-struc-1}, we have
$[u,v]F\in \A_n(D_j/{u \to v})$ for $F\in \A_n(D_j)$.
Then, if $D_i/{u \to v}\neq D_j/{u \to v}$, we can get
\[
L(D_i)(F)=L(D_i/{u\to v})[u,v](F)=L(D_i/{u\to v})(F')=0.
\]
Here $F'\in \A_n(D_j/{u \to v})$ and the last equality holds by the induction hypothesis.
Hence, the lemma follows by showing that $D_i/{u\to v}\neq D_j/{u\to v}$.

We then consider two cases (recall that $u$ is not a root of $D_i$ and $D_j$):
i) If $u$ is not a leaf of $D_j$, then it has at least one child, say $u'>u$ since $D_j$ is nearly increasing. It follows that $u'$ becomes a child of $v$ in $D_j/{u\to v}$, which can not be a child of $v$ in $D_i/{u\to v}$ by our choice of $u$.
ii) If $u$ is a leaf in $D_j$, we assume to the contrary that $D_i/{u\to v}=D_j/{u\to v}$.
Then $u$ is also the largest child of $v$ in $D_j$. It follows that $D_i= D_j$, a contradiction.

Equation~\eqref{e-orthogonality-n-2} follows from the proof above and Corollary~\ref{cor-eq}.
\end{proof}

We remark that the condition that $D_i$ and $D_j$ have the same roots cannot be dropped:
consider the case when $D_i$ and $D_j$ are both one-edge graphs consisting of $2\to 3$
and $3\to 2$ respectively.

Recall that $\Par(N,\ell)$ is the set of all partitions of $N=\{1,\dots,n\}$ with $\ell$ blocks.
\def\spann{\texttt{span}}
For $s=0,1,\dots,n-2$, let $\S \in \Par(N,n-s)$ and the integer sequence $\r$ be compatible with $\S$.
Denote by $B_{\S}:=\{L(D): D\in \Daugn_{\S}\}$ and $B_{\S,\r}:=\{L(D): D\in \Daugn_{\S,\r}\}$.
We have the next necessary and sufficient condition for the judgement of the zero operator in the spanning space of $B_{\S,\r}$, and obtain a basis of it.
\begin{cor}\label{cor-zero-operator}
For any $\L\in \spann (B_{\S,\r})$, $\L$ is the zero operator in $\CTA$
if and only if $\L \epsilon(D)=0$ for all $D\in \Daugn_{\S,\r}$.
The elements of $B_{\S,\r}$ form a basis of $ \spann (B_{\S,\r})$.
\end{cor}
\begin{proof} We only prove the sufficiency of the first claim, the necessity is trivial.

Since $\L\in \spann (B_{\S,\r})$, we can write
\begin{equation}\label{e-zero-operator}
\L:= \sum_{D \in \Daugn_{\S,\r}} c_D L(D)
\end{equation}
for some complex numbers $c_D$. It suffices to show that $c_{D'}=0$ for each particular $D' \in \Daugn_{\S,\r}$.
Applying both sides of \eqref{e-zero-operator} to
$\epsilon(D')$ and using Lemma~\ref{lem-orthogonality-2} gives
$\L \epsilon(D')=c_{D'}=0$.

Assume
\[
\sum_{D \in \Daugn_{\S,\r}} c_D L(D)=0.
\]
From the proof above, we can obtain every $c_D=0$.
Then the elements of $B_{\S,\r}$ are linearly independent and hence form a basis of
$\spann (B_{\S,\r})$.
\end{proof}

\begin{comment}
Now we prove the following structural results.
\red{XXXXXXXXXXX}

\begin{lem}\label{lem-struct-S,r}
For a fixed $s \in \{0,1,\dots, n-2\}$, let $\S \in \Par(N,n-s)$ and $\r$ be compatible with $\S$. Then $B_{\S,\r}$ is a basis of $\spann (B_{\S,\r})$.
\end{lem}
\begin{proof}
The proof follows directly from Proposition \ref{prop-span} and Corollary \ref{cor-zero-operator}.

\red{Not quite directly!}
\end{proof}
\end{comment}

From the space $\spann(B_{\S,\r})$, we can construct $\CTA$ through direct sum. In other words,
the algebra $\CTA$ is a direct sum of the spanning spaces $\spann(B_{\S,\r})$ for certain $\S$ and $\r$.
We obtain this step by step in the next three decomposition lemmas.
\begin{lem}\label{lem-struct-S}
For fixed $s \in \{0,\dots, n-2\}$ and $\S \in \Par(N,n-s)$,
\begin{equation}\label{e-direct}
\spann(B_{\S})= \bigoplus_{\r \compat \S} \spann \big(B_{\S,\r}\big).
\end{equation}%, where the sum is over all possible choices of $\r$ compatible with $\S$.
%such that $B_{\S,\r}$ are distinct.
\end{lem}
\begin{proof}
Assume
\begin{equation}\label{e-sum-operator}
\sum_{\r} \L_{\r}=0,
\end{equation}
where $\r$ is over all integer sequences compatible with $\S$
and $\L_{\r}\in \spann(B_{\S,\r})$. We need to show that every $\L_{\r}$ is in fact the zero operator.

Assume to the contrary that there exists a $\r'=(r_1',r_2',\dots, r_k')$ such that $\L_{\r'}$ is nonzero. Note that $k=n-s\ge 2$.
By Corollary \ref{cor-zero-operator}, there exists a $D'\in \Daugn_{\S,\r'}$ such that $\L_{\r'}\epsilon(D')=\gamma_{\r',\r'}\in \mathbb{C}\setminus \{0\}$.
By Corollary \ref{cor-eq} we have $\L_{\r} \epsilon(D')=\gamma_{\r,\r'}\in \mathbb{C}.$

Now applying both sides of \eqref{e-sum-operator} to $x_{r_1'}^{1-k}x_{r_2'}\cdots x_{r_k'}\epsilon(D')$ and using Corollary~\ref{cor-ortho-1}, we obtain
\[
\sum_{\r} \gamma_{\r,\r'} x_{r_1}^{1-k}x_{r_2}\cdots x_{r_k}=0,
\]
where $\r=(r_1,\dots,r_k)$ is also over all integer sequences compatible with $\S$.
By the linear independency of the monomials, we have $\gamma_{\r,\r'}=0$ for all $\r$. In particular $\gamma_{\r',\r'}=0$, a contradiction.
\end{proof}

\begin{lem}\label{lem-struct-s}
(1) The set $B^{n-1}=\{L(D): D\in \Taug_{n-1}\}$
is a basis of the constant term algebra $\CTA^{n-1}$;

(2) For $s=0,1,\dots,n-2$,
\begin{equation}\label{e-direct-s}
\CTA^s= \bigoplus_{\S \in \Par(N,n-s)} \spann \big(B_{\S}\big)=\spann (B^s),
\end{equation}
where $B^s =\{L(D): D\in \Daugn_{s}\}.$
\end{lem}
\begin{proof}
By Proposition~\ref{prop-span}, the algebra $\CTA^{n-1}$ is spanned by $B^{n-1}$.
Thus, to prove (1), it suffices to prove that the elements of $B^{n-1}$ are linearly independent. Assume
\begin{equation}\label{e-mainproof-1}
\sum_{T \in \Taug_{n-1}} c_T L(T) =0,
\end{equation}
where %the sum ranges over all augmented increasing rooted trees $T_i$ (with root $1$), and
$c_T \in \mathbb{C}$ for all $T$. Applying both sides of \eqref{e-mainproof-1} to
$\epsilon(T')$ for some $T'\in \Taug_{n-1}$ and using Lemma \ref{lem-orthogonality-2},
we get $c_{T'}=0$. Let $T'$ range over all trees of $\Taug_{n-1}$. This gives that all the coefficients $c_T$ in \eqref{e-mainproof-1} vanish. Then, the $L(T)$'s are linearly independent.

%\begin{equation}\label{e-mainproof-1}
%\sum_{T_i \in \Taug_{n-1}} c_i L(T_i) =0,
%\end{equation}
%where %the sum ranges over all augmented increasing rooted trees $T_i$ (with root $1$), and
%$c_i \in \mathbb{C}$ for all $i$. Applying both sides of \eqref{e-mainproof-1} to
%$\epsilon(T_{j})$ for some $T_j\in \Taug_{n-1}$ and using Lemma \ref{lem-orthogonality-2},
%we get $c_{j}=0$. Let $T_j$ ranges over all trees of $\Taug_{n-1}$. This gives that all the coefficients $c_i$ in \eqref{e-mainproof-1} vanish. Then, the $L(T_i)$ are linearly independent.

To prove (2), suppose
\begin{equation}\label{e-mainproof-2}
\sum_{\S \in \Par(N,n-s)} \L_{\S}=0,
\end{equation}
where $\L_{\S} \in \spann \big(B_{\S}\big)$. We need to show
that $\L_{\S}$ is the zero operator for each $\S$.

Assume to the contrary that there exists an $\S'$ such that
$\L_{\S'}\neq 0$. Then by Corollary~\ref{cor-zero-operator} and Lemma \ref{lem-struct-S}, there exists a $D'\in \Daugn_{\S}$ such that
$\L_{\S'} \epsilon(D')=\gamma \in \mathbb{C}\setminus \{0\}$.
By Corollary \ref{cor-ortho-1}, $\L_{\S} \epsilon(D')=0$ for all
$\S\neq \S'$. Therefore,
applying both sides of \eqref{e-mainproof-2} to $\epsilon(D')$ gives
$\gamma=0$, a contradiction.
\end{proof}

\begin{lem}\label{lem-struct-CTA}
The constant term algebra $\CTA$ has the following direct sum decomposition:
\begin{equation}\label{e-direct2}
\CTA = \bigoplus_{s=0}^{n-1}\CTA^{s} = \bigoplus_{s=0}^{n-1} \spann(B^s).
\end{equation}
\end{lem}
\begin{proof}
Suppose
\begin{equation}\label{e-last}
\sum_{s=0}^{n-1} \L_s=0,
\end{equation}
where $\L_s\in \CTA^s$.
We need to show that $\L_s$ is the zero operator for each $s$.

Assume the contrary. Let $s'$ be the smallest integer such that $\L_{s'}\neq 0$.
By Corollary~\ref{cor-zero-operator}, Lemma~\ref{lem-struct-S} and Lemma \ref{lem-struct-s}, there exists a $D'\in \Daugn_{s'}$ when $s'\leq n-2$ or $D'\in \Taug_{s'}$ when $s'=n-1$,
 such that
$\L_{s'} \epsilon(D')=\gamma' \in \mathbb{C}\setminus \{0\}$.
By Corollary \ref{cor-ge}, we have $\L_{s} \epsilon(D')=0$ for all $s>s'$.

Now by applying both sides of \eqref{e-last} to $\epsilon(D')$, we have
\[
0=\sum_{s=0}^{n-1} \L_s \epsilon(D')=\gamma'.
\]
This is a contradiction.
\end{proof}

Finally, we obtain our main theorem.
\begin{thm}\label{thm-struct}
The union $\cup_{s=0}^{n-1}B^s$ forms a basis of $\CTA$.
\end{thm}
\begin{proof}
This is a direct consequence of Lemmas \ref{lem-struct-CTA}, \ref{lem-struct-s}, \ref{lem-struct-S} and Corollary~\ref{cor-zero-operator}.
\end{proof}

\section{Concluding remark}
We have established a structural result of the constant term algebra of type $A$. This algebra shall be useful in computing
constant terms of type $A$ rational functions.
As an application, the algorithm based on this algebra will be efficient in the computation of the Ehrhart polynomial $H_n(t)$ in \eqref{e-intro-2}.

In the next sequel, at least for the rational functions $F$ arising from the computation of $H_n(t)$, we will express $\CT\limits_{\x} F$ as a linear combination of
$LF$ (where $L \in \CTA^{n-1}$)
%$\CT\limits_{T} F=[i_{n-1},j_{n-1}]\cdots [i_1,j_1] F$
by iterated Laurent series.
Then we will simplify the computation by our structural result.
In this way, we have obtained simple formulas of $H_n(t)$ for $n\le 6$ so far.
It is hopeful that we can solve the open problem of computing $H_{10}(t)$ by this idea.

It is also natural to consider other types of rational functions. For example,
the type $D$ rational functions: the denominator factors are of the forms $1-x_i/x_j$ and $1-x_ix_j$. There should exist a similar algebra.

Another further direction is to consider the $q$-analogues of type $A$ rational functions. We anticipate a similar theory for the $q$-Dyson type constant terms.

\textbf{Acknowledgments:}
This work was supported by the National Natural Science Foundation of China (No. 12071311, 12171487).

\end{document}